\documentclass[11pt]{amsart}
\usepackage[latin9]{inputenc}
\usepackage{amsthm}
\usepackage{amstext}
\usepackage{amssymb}
\usepackage{graphicx}
\usepackage{esint}

\makeatletter

\providecommand{\tabularnewline}{\\}

\numberwithin{equation}{section}
\numberwithin{figure}{section}

\usepackage{amsfonts}
\providecommand{\U}[1]{\protect\rule{.1in}{.1in}}
\newtheorem{theorem}{Theorem}[section]

\newtheorem{corollary}[theorem]{Corollary}

\theoremstyle{definition}
\newtheorem{definition}[theorem]{Definition}
\theoremstyle{remark}
\newtheorem{example}[theorem]{Example}
\newtheorem{remark}[theorem]{Remark}
\numberwithin{equation}{section}
\DeclareMathOperator{\sech}{sech}

\makeatother

\title[SPPS for polynomial pencils]{Spectral parameter power series for polynomial pencils of Sturm-Liouville
operators and Zakharov-Shabat systems}

\author[V. V. Kravchenko]{Vladislav V. Kravchenko}
\author[S. M. Torba]{Sergii M. Torba}
\author[U. Velasco-García]{Ulises Velasco-García}

\address[V. V. Kravchenko, S. M. Torba and U. Velasco-García]{Departamento
de Matemáticas, CINVESTAV del IPN, Unidad Querétaro, Libramiento Norponiente
2000, Fracc. Real de Juriquilla, Querétaro, Qro. C.P. 76230 MEXICO.}

\email[V. V. Kravchenko]{vkravchenko@math.cinvestav.edu.mx}
\email[S. M. Torba]{storba@math.cinvestav.edu.mx}
\email[U. Velasco-García]{ulisesv@math.cinvestav.edu.mx}

\thanks{Research of the authors was partially supported by
CONACYT, Mexico via the project 166141.}

\subjclass[2010]{Primary 34A45, 34B07, 34L16, 41A58, 65L10, 65L15; Secondary 34B09, 34L40, 35Q41, 35Q55, 35Q74, 81Q05}
\keywords{Sturm-Liouville problems, SPPS representation, energy dependent potential, polynomial pencil of Sturm-Liouville operators, Zakharov-Shabat system, dispersion equation, numerical method}

\begin{document}
\begin{abstract}
A spectral parameter power series (SPPS) representation for solutions
of Sturm-Liouville equations of the form
\[
\left(pu'\right)'+qu=u\sum_{k=1}^{N}\lambda^{k}r_{k}\eqno(\ast)
\]
is obtained. It allows one to write a general solution of the equation
as a power series in terms of the spectral parameter $\lambda$. The
coefficients of the series are given in terms of recursive integrals
involving a particular solution of the equation $\left(pu_{0}'\right)'+qu_{0}=0$.
The convenient form of the solution of ($\ast$) provides an efficient
numerical method for solving corresponding initial value, boundary
value and spectral problems.

A special case of the Sturm-Liouville equation ($\ast$) arises in
relation with the Zakharov-Shabat system. We derive an SPPS representation
for its general solution and consider other applications as the one-dimensional
Dirac system and the equation describing a damped string. Several
numerical examples illustrate the efficiency and the accuracy of the
numerical method based on the SPPS representations which besides its
natural advantages like the simplicity in implementation and accuracy
is applicable to the problems admitting complex coefficients, spectral
parameter dependent boundary conditions and complex spectrum.
\end{abstract}

\maketitle

\section{Introduction}

Equations of the form
\begin{equation}
\left(pu^{\prime}\right)^{\prime}+qu=u\sum_{k=1}^{N}\lambda^{k}r_{k}\label{SLpencil}
\end{equation}
arise in numerous applications and have been studied in a considerable
number of publications where sometimes they were referred to as Schrödinger-type
equations with a polynomial energy-dependent potential \cite{Scrodinger-type equation2,Scrodinger-type equation1},
Schrödinger equations with an energy-dependent potential \cite{Lin Li Qian},
Schrödinger operators with energy depending potentials \cite{Laptev Shterenberg Sukhanov},
Klein-Gordon s-wave equations \cite{Baramov}, energy-dependent Sturm--{}Liouville
equations \cite{Energy dependent S-L}, pencil of Sturm-Liouville
operators \cite{Pencil of Sturm-Liouville,Pencil of SL}, polynomial
pencil of Sturm-Liouville equations \cite{pecil of S-L} or bundle
for Sturm-Liouville operators \cite{Bundle of Sturm-Liouville}. Here
$p$, $q$ and $r_{k}$, $k=1,\ldots,N$ are complex-valued functions
of the real variable $x$ and the complex number $\lambda$ is the
spectral parameter. Additional conditions on the coefficients of equation
(\ref{SLpencil}) will be formulated below.

In the present work we obtain a spectral parameter power series (SPPS)
representation for solutions of (\ref{SLpencil}). This result generalizes
what was done for the Sturm-Liouville equation
\begin{equation}
\left(pu^{\prime}\right)^{\prime}+qu=\lambda ru\label{eq: Sturm-Liouville}
\end{equation}
in \cite{KrCV08} (see also \cite{Libro Kravchenko,SPPS}).

The general solution of (\ref{SLpencil}) is represented in the form
of a power series in terms of the spectral parameter $\lambda$. The
coefficients of the series are calculated as recursive integrals in
terms of a particular solution of the equation
\begin{equation}
\left(pu_{0}^{\prime}\right)^{\prime}+qu_{0}=0.\label{SL0}
\end{equation}
Moreover, to construct the particular solution one can use the same
SPPS approach as was explained in \cite{Libro Kravchenko,SPPS}.

As was shown in a number of recent publications the SPPS representation
provides an efficient and accurate method for solving initial value,
boundary value and spectral problems (see \cite{CKKO2009}, \cite{CKOR},
\cite{ErbeMertPeterson2012}, \cite{KKB2013}, \cite{KiraRosu2010},
\cite{KKR2012}, \cite{Khmelnytskaya Serroukh 2013 MMAS}, \cite{Khmelnytskaya Torchinska 2010},
\cite{KrCV08}, \cite{Libro Kravchenko}, \cite{SPPS}, \cite{Kravchenko-Velasco},
\cite{Rabinovich et al 2013 MMAS}). In this paper we demonstrate
this fact in application to equation (\ref{SLpencil}). The main advantage
of the SPPS approach applied to spectral problems consists in the
possibility to write down a characteristic (or dispersion) equation
of a problem in an explicit form $\Phi(\lambda)=0$ where the function
$\Phi$ is analytic with respect to $\lambda$. Due to the SPPS representation
the function $\Phi$ is given in the form of a Taylor series whose
coefficients can be computed in terms of the particular solution of
(\ref{SL0}). For a numerical computation a partial sum of the Taylor
series $\Phi_{M}$ is calculated and the approximate solution of the
spectral problem reduces to the problem of finding zeros of the polynomial
$\Phi_{M}(\lambda)$. The accuracy of this approach depends on the
smallness of the difference $\left\vert \Phi-\Phi_{M}\right\vert $
for which we present estimates and additionally show that due to Rouche's
theorem a good approximation of $\Phi$ by $\Phi_{M}$ in a domain
$\Omega\subset\mathbb{C}$ guarantees that no root of $\Phi_{M}$
which would correspond to a fictitious eigenvalue can appear in $\Omega$
as a result of truncation of the spectral parameter power series.
That is, in a domain $\Omega$ where $\max\left\vert \Phi-\Phi_{M}\right\vert $
is sufficiently small all roots of $\Phi_{M}$ approximate true eigenvalues
of the spectral problem. We show that for localizing the eigenvalues
and for estimating their number in a complex domain the argument principle
can be used.

The Zakharov-Shabat system is one of the physical models which can
be reduced to an equation of the form (\ref{SLpencil}). When the
potential in the Zakharov-Shabat system is real valued the situation
is simpler in the sense that the system reduces to an equation (\ref{eq: Sturm-Liouville})
and in that case the SPPS representations were obtained and used for
solving spectral problems in \cite{Kravchenko-Velasco}. Nevertheless
when the potential is complex valued such reduction in general is impossible
and an equation of the form (\ref{SLpencil}) with $N=2$ naturally
arises. We derive an SPPS representation for solutions of such general
Zakharov-Shabat system as well as an analytic form for the dispersion
equation of the corresponding spectral problem in the case of a compactly
supported potential. The numerical implementation of the method together
with the argument principle is illustrated on a couple of test problems.
We show that the one-dimensional Dirac system can be studied in a
similar way.

Another physical model considered in this paper and reducing to an
equation of the form (\ref{SLpencil}) again with $N=2$ corresponds
to a damped string \cite{Atkinson,Jaulent,Kobyakova}. We obtain an
analytic form for the dispersion equation of the corresponding spectral
problem and illustrate the application of the SPPS method on a number
of numerical examples.

\section{SPPS representation for solutions of (\ref{SLpencil})}

In this section we derive and prove an SPPS representation for the
general solution of (\ref{SLpencil}).

\begin{theorem} \label{thm:SL grado N} Suppose that on a finite
segment $\left[a,b\right]$ the equation
\begin{equation*}
\left(pv^{\prime}\right)^{\prime}+qv=0\label{eq: S-L lamda0 N}
\end{equation*}
possesses a particular non-vanishing solution $u_{0}$ such that the
functions $u_{0}^{2}r_{k}$, $k=1,\ldots,N$ and $\frac{1}{u_{0}^{2}p}$
are continuous on $\left[a,b\right]$. Then a general solution of
the equation
\begin{equation}
\left(pu^{\prime}\right)^{\prime}+qu=u\sum_{k=1}^{N}\lambda^{k}r_{k}\label{eq: S-L Bundle}
\end{equation}
has the form $u=c_{1}u_{1}+c_{2}u_{2}$, where $c_{1}$ and $c_{2}$
are arbitrary complex constants,
\begin{equation}
u_{1}=u_{0}\sum_{n=0}^{\infty}\lambda^{n}\widetilde{X}^{\left(2n\right)}\quad\text{ and }\quad u_{2}=u_{0}\sum_{n=0}^{\infty}\lambda^{n}X^{\left(2n+1\right)}\label{eq:Sumas-N}
\end{equation}
with $\widetilde{X}^{\left(n\right)}$ and $\ X^{\left(n\right)}$
being defined by the recursive relations
\begin{align*}
\widetilde{X}^{\left(n\right)} & \equiv X^{\left(n\right)}\equiv0\quad\text{ for }n<0,\\
\widetilde{X}^{\left(0\right)} & \equiv X^{\left(0\right)}\equiv1,
\end{align*}
and
\begin{align}
\widetilde{X}^{\left(n\right)}\left(x\right) & =\begin{cases}
\intop_{x_{0}}^{x}u_{0}^{2}\left(y\right)\sum_{k=1}^{N}\widetilde{X}^{\left(n-2k+1\right)}\left(y\right)r_{k}\left(y\right)dy & n\text{-odd}\\
\intop_{x_{0}}^{x}\widetilde{X}^{\left(n-1\right)}\left(y\right)\frac{1}{u_{0}^{2}\left(y\right)p\left(y\right)}dy & n\text{-even,}
\end{cases}\label{eq: equis tilde}\\
X^{\left(n\right)}\left(x\right) & =\begin{cases}
\intop_{x_{0}}^{x}X^{\left(n-1\right)}\left(y\right)\frac{1}{u_{0}^{2}\left(y\right)p\left(y\right)}dy & n\text{-odd}\\
\intop_{x_{0}}^{x}u_{0}^{2}\left(y\right)\sum_{k=1}^{N}X^{\left(n-2k+1\right)}\left(y\right)r_{k}\left(y\right)dy & n\text{-even,}
\end{cases}\label{eq: equis sin tilde}
\end{align}
where $x_{0}$ is an arbitrary point of $\left[a,b\right]$ such that
$p(x_{0})\neq0$. Both series in \eqref{eq:Sumas-N} converge uniformly
on $\left[a,b\right]$. \end{theorem}

\begin{proof} We prove first that $u_{1}$ and $u_{2}$ are indeed
solutions of (\ref{eq: S-L Bundle}) whenever the application of the
operator $L=\frac{d}{dx}p\frac{d}{dx}+q$ makes sense. If $u_{0}$
is a non-vanishing solution of $Lu_{0}=0$, then the operator $L=\frac{d}{dx}p\frac{d}{dx}+q$
admits the Pólya factorization $L=\frac{1}{u_{0}}\partial pu_{0}^{2}\partial\frac{1}{u_{0}}$.
Application of $L$ to $u_{1}$ using (\ref{eq: equis tilde}) gives
\begin{align*}
Lu_{1} & =\frac{1}{u_{0}}\partial pu_{0}^{2}\partial\sum_{k=0}^{\infty}\lambda^{k}\widetilde{X}^{\left(2k\right)}=\frac{1}{u_{0}}\sum_{k=1}^{\infty}\lambda^{k}\partial\widetilde{X}^{\left(2k-1\right)}\\
 & =\frac{1}{u_{0}}\sum_{k=1}^{\infty}\lambda^{k}\sum_{n=1}^{N}\widetilde{X}^{\left(2\left(k-n\right)\right)}u_{0}^{2}r_{n}=\sum_{n=1}^{N}r_{n}u_{0}\sum_{k=1}^{\infty}\lambda^{k}\widetilde{X}^{\left(2\left(k-n\right)\right)}\\
 & =\sum_{n=1}^{N}r_{n}u_{0}\sum_{k=n}^{\infty}\lambda^{k}\widetilde{X}^{\left(2\left(k-n\right)\right)}=\sum_{n=1}^{N}r_{n}u_{0}\sum_{k=0}^{\infty}\lambda^{k+n}\widetilde{X}^{\left(2k\right)}\\
 & =\sum_{n=1}^{N}r_{n}\lambda^{n}u_{0}\sum_{k=0}^{\infty}\lambda^{k}\widetilde{X}^{\left(2k\right)}=u_{1}\sum_{n=1}^{N}\lambda^{n}r_{n}.
\end{align*}
Similarly, application of $L$ to $u_{2}$ shows that $u_2$ satisfies (\ref{eq: S-L Bundle}).

In order to give sense to this chain of equalities it is sufficient
to prove the uniform convergence of the series involved in $u_{1}$
and $u_{2}$ as well as the uniform convergence of the series obtained
by a term-wise differentiation of $\frac{u_{1}}{u_{0}}$ and $\frac{u_{2}}{u_{0}}$.
First, we prove the uniform convergence of the series involved in
$u_{1}$. This can be done with the aid of the Weierstrass M-test.

We prove by induction that for $n\ge0$ the inequality
\begin{equation}
\left\vert \widetilde{X}^{\left(2n\right)}(x)\right\vert \leq\sum_{k=0}^{n-\left[\frac{n}{N}\right]}\binom{n}{k}\frac{\left(m\left\vert x-x_{0}\right\vert \right)^{2\left(n-k\right)}}{\left(2\left(n-k\right)\right)!}\label{eq: Mayorante para X tilde}
\end{equation}
is valid, where $m$ denotes the maximum of the functions $\left\vert u_{0}^{2}(x)r_{k}(x)\right\vert $, $k=1,\ldots,N$
and $\bigl\vert \frac{1}{u_{0}^{2}\left(x\right)p\left(x\right)}\bigr\vert $
on $[a,b]$. For $n=0$, $\vert \widetilde{X}^{\left(0\right)}\vert =1$,
and hence (\ref{eq: Mayorante para X tilde}) holds. For the inductive
step we assume that (\ref{eq: Mayorante para X tilde}) holds for
some $n$ and prove that for $n+1$ the inequality
\begin{equation}\label{eq: estimate}
\left\vert \widetilde{X}^{\left(2\left(n+1\right)\right)}(x)\right\vert \leq\sum_{k=0}^{n+1-\left[\frac{n+1}{N}\right]}\binom{n+1}{k}\frac{\left(m\left\vert x-x_{0}\right\vert \right)^{2\left(n+1-k\right)}}{\left(2\left(n+1-k\right)\right)!}
\end{equation}
holds. Suppose that $x_{0}\leq x$ (the opposite case is similar),
recalling the definition of $m$ and (\ref{eq: equis tilde}) we have
\begin{align*}
\left\vert \widetilde{X}^{\left(2\left(n+1\right)\right)}(x)\right\vert  & =\left\vert \intop_{x_{0}}^{x}\widetilde{X}^{\left(2n+1\right)}\left(y\right)\frac{1}{u_{0}^{2}\left(y\right)p\left(y\right)}dy\right\vert \leq m\intop_{x_{0}}^{x}\left\vert \widetilde{X}^{\left(2n+1\right)}\left(y\right)\right\vert dy\\
 & \le m\intop_{x_{0}}^{x}\intop_{x_{0}}^{y}\sum_{j=1}^{N}\left\vert \widetilde{X}^{\left(2\left(n-j+1\right)\right)}\left(z\right)u_{0}^{2}\left(z\right)r_{j}\left(z\right)\right\vert dzdy\\
 & \leq m^{2}\intop_{x_{0}}^{x}\intop_{x_{0}}^{y}\sum_{j=1}^{N}\left\vert \widetilde{X}^{\left(2\left(n-j+1\right)\right)}\left(z\right)\right\vert dzdy=m^{2}\intop_{x_{0}}^{x}\intop_{x_{0}}^{y}\sum_{j=1}^{\min\{n+1,N\}}\left\vert \widetilde{X}^{\left(2\left(n-j+1\right)\right)}\left(z\right)\right\vert dzdy.
\end{align*}
Applying (\ref{eq: Mayorante para X tilde}) one obtains
\begin{align*}
\left\vert \widetilde{X}^{\left(2\left(n+1\right)\right)}(x)\right\vert  & \leq m^{2}\intop_{x_{0}}^{x}\intop_{x_{0}}^{y}\sum_{j=1}^{\min\{n+1,N\}}\sum_{k=0}^{n-j+1-\left[\frac{n-j+1}{N}\right]}\binom{n-j+1}{k}\frac{\left(m\left\vert z-x_{0}\right\vert \right)^{2\left(n-j-k+1\right)}}{\left(2\left(n-j-k+1\right)\right)!}dzdy\\
 & =\sum_{j=1}^{\min\{n+1,N\}}\sum_{k=0}^{n-j+1-\left[\frac{n-j+1}{N}\right]}\frac{\binom{n-j+1}{k}m^{2\left(n-j-k+2\right)}}{\left(2\left(n-j-k+1\right)\right)!}\intop_{x_{0}}^{x}\intop_{x_{0}}^{y}\left\vert z-x_{0}\right\vert ^{2\left(n-j-k+1\right)}dzdy\\
 & =\sum_{j=1}^{\min\{n+1,N\}}\sum_{k=0}^{n-j+1-\left[\frac{n-j+1}{N}\right]}\binom{n-j+1}{k}\frac{\left(m\left\vert x-x_{0}\right\vert \right)^{2\left(n-j-k+2\right)}}{\left(2\left(n-j-k+2\right)\right)!}.
\end{align*}
It is easy to see that from $1\leq j\leq\min\{n+1,N\}$ and $0\leq k\leq n-j+1-\left[\frac{n-j+1}{N}\right]$
it follows that
\[
0\leq k+j-1\leq n+1-\left[\frac{n+1}{N}\right].
\]
We rearrange the terms with respect to $l=k+j-1$,
\begin{align*}
 & \sum_{j=1}^{\min\{n+1,N\}}\sum_{k=0}^{n-j+1-\left[\frac{n-j+1}{N}\right]}\binom{n-j+1}{k}\frac{\left(m\left\vert x-x_{0}\right\vert \right)^{2\left(n-j-k+2\right)}}{\left(2\left(n-j-k+2\right)\right)!}\\
 & \qquad=\sum_{l=0}^{n+1-\left[\frac{n+1}{N}\right]}\sum_{j=1}^{\min\{N,n+1,l+1\}}\binom{n-(j-1)}{l-(j-1)}\frac{\left(m\left\vert x-x_{0}\right\vert \right)^{2\left(n+1-l\right)}}{\left(2\left(n+1-l\right)\right)!}\\
 & \qquad\leq\sum_{l=0}^{n+1-\left[\frac{n+1}{N}\right]}\frac{\left(m\left\vert x-x_{0}\right\vert \right)^{2\left(n+1-l\right)}}{\left(2\left(n+1-l\right)\right)!}\sum_{j=0}^{l}\binom{n-j}{l-j}.
\end{align*}
Using the well known combinatorial relation
\[
\binom{n+1}{l}=\sum_{j=0}^{l}\binom{n-j}{l-j}\qquad\text{ for }l<n+1
\]
we obtain
\[
\sum_{l=0}^{n+1-\left[\frac{n+1}{N}\right]}\frac{\left(m\left\vert x-x_{0}\right\vert \right)^{2\left(n+1-l\right)}}{\left(2\left(n+1-l\right)\right)!}\sum_{j=0}^{l}\binom{n-j}{l-j}=\sum_{l=0}^{n+1-\left[\frac{n+1}{N}\right]}\frac{\left(m\left\vert x-x_{0}\right\vert \right)^{2\left(n+1-l\right)}}{\left(2\left(n+1-l\right)\right)!}\binom{n+1}{l},
\]
which finishes the proof of \eqref{eq: estimate} and hence of the estimate \eqref{eq: Mayorante para X tilde} for any integer $n\ge 0$.

Next we prove that the series $\sum_{n=0}^{\infty}\lambda^{n}\widetilde{X}^{\left(2n\right)}$
converges uniformly on $\left[a,b\right]$. We have for $n\geq0$
\begin{align*}
\left\vert \widetilde{X}^{\left(2n\right)}(x)\right\vert  & \leq\sum_{k=0}^{n-\left[\frac{n}{N}\right]}\binom{n}{k}\frac{\left(m\left\vert x-x_{0}\right\vert \right)^{2\left(n-k\right)}}{\left(2\left(n-k\right)\right)!}\leq\sum_{k=0}^{n-\left[\frac{n}{N}\right]}\binom{n}{k}\frac{\left(m\left(b-a\right)\right)^{2\left(n-k\right)}}{\left(2\left[\frac{n}{N}\right]\right)!}\\
 & \leq\frac{1}{\left(2\left[\frac{n}{N}\right]\right)!}\sum_{k=0}^{n}\binom{n}{k}\left(m\left(b-a\right)\right)^{2\left(n-k\right)}1^{k}=\frac{\left(\left(m\left(b-a\right)\right)^{2}+1\right)^{n}}{\left(2\left[\frac{n}{N}\right]\right)!}.
\end{align*}
Hence
\begin{equation*}
\sum_{n=0}^{\infty}\left\vert \lambda\right\vert ^{n}\left\vert \widetilde{X}^{\left(2n\right)}\right\vert  \leq\sum_{n=0}^{\infty}\left\vert \lambda\right\vert ^{n}\frac{\left(\left(m\left(b-a\right)\right)^{2}+1\right)^{n}}{\left(2\left[\frac{n}{N}\right]\right)!},
\end{equation*}
and grouping terms with respect to $n_1=[\frac nN]$ we obtain
\begin{equation*}
\sum_{n=0}^{\infty}\left\vert \lambda\right\vert ^{n}\left\vert \widetilde{X}^{\left(2n\right)}\right\vert \le \sum_{n_1=0}^{\infty}\left\vert \lambda\right\vert ^{n_1N}\frac{\left(\left(m\left(b-a\right)\right)^{2}+1\right)^{n_1N}}{\left(2n_1\right)!}\sum_{k=0}^{N-1}\left\vert \lambda\right\vert ^{k}\left(\left(m\left(b-a\right)\right)^{2}+1\right)^{k}.
\end{equation*}
Considering the notation $M=\left\vert \lambda\right\vert \left(\left(m\left(b-a\right)\right)^{2}+1\right)$
one obtains
\[
\sum_{n=0}^{\infty}\left\vert \lambda\right\vert ^{n}\left\vert \widetilde{X}^{\left(2n\right)}\right\vert \leq\sum_{k=0}^{N-1}M^{k}\sum_{n_1=0}^{\infty}\frac{M^{n_1N}}{(2n_1)!}=\sum_{k=0}^{N-1}M^{k}\cosh M^{N/2}.
\]
Then, by the Weierstrass M-test the series $\sum_{n=0}^{\infty}\lambda^{n}\widetilde{X}^{\left(2n\right)}$
converges uniformly on $\left[a,b\right]$.

The proof for the function $u_{2}$ and the derivatives $\left(\frac{u_{1}}{u_{0}}\right)^{\prime}=\frac{1}{u_{0}p}\sum_{n=0}^{\infty}\lambda^{n+1}\widetilde{X}^{(2n+1)}$
and $\left(\frac{u_{2}}{u_{0}}\right)^{\prime}=\frac{1}{u_{0}p}\sum_{n=0}^{\infty}\lambda^{n}X^{(2n)}$
is similar. The last step is to verify that the Wronskian of $u_{1}$
and $u_{2}$ is different from zero at least at one point (which necessarily
implies the linear independence of $u_{1}$ and $u_{2}$ on the whole
segment $[a,b]$). It is easy to see that by definition all $\widetilde{X}^{(n)}(x_{0})$
and $X^{(n)}(x_{0})$ vanish except $\widetilde{X}^{(0)}(x_{0})=X^{(0)}(x_{0})=1$.
Thus
\begin{align}
u_{1}\left(x_{0}\right) & =u_{0}\left(x_{0}\right), & u_{1}^{\prime}\left(x_{0}\right) & =u_{0}^{\prime}\left(x_{0}\right),\label{eq:sols en x0 1}\\
u_{2}\left(x_{0}\right) & =0, & u_{2}^{\prime}\left(x_{0}\right) & =\frac{1}{u_{0}\left(x_{0}\right)p\left(x_{0}\right)},\label{eq:sols en x0 2}
\end{align}
and hence the Wronskian of $u_{1}$ and $u_{2}$ at $x_{0}$ is $\frac{1}{p\left(x_{0}\right)}\neq0$.
\end{proof}

\begin{remark} The particular solution $u_{0}$ which in general
can be complex valued may be constructed by means of the SPPS
representation as well (see \cite{Libro Kravchenko,SPPS}). \end{remark}

When $N=1$, the result of Theorem \ref{thm:SL grado N} reduces to
the SPPS representation for solutions of a classic Sturm-Liouville
equation presented in \cite{SPPS}.

\begin{example} Consider the differential equation
\[
y^{\prime\prime}=y(\lambda+2\lambda^{2})
\]
on the segment $\left[0,1\right]$ with the initial conditions
\begin{equation}
y\left(0\right)=1\quad\text{and}\quad y^{\prime}\left(0\right)=0.\label{example initial cond}
\end{equation}
The unique solution of this problem has the form $y(x)=\cosh\left(x\sqrt{\lambda+2\lambda^{2}}\right)$.
This is an equation of the form (\ref{eq: S-L Bundle}) with $p=1$,
$q=0$ and $r_{k}=k$, $k=1,2$. Let us consider its solution in terms
of the SPPS representation from Theorem \ref{thm:SL grado N}. Taking
$x_{0}=0$ and $u_{0}\equiv1$ as a particular solution of the differential
equation $y^{\prime\prime}=0$ we compute the functions $\widetilde{X}^{\left(2n\right)}$
for $n=1,2,3$ to find the first terms of the series $u_{1}=u_{0}\sum_{k=0}^{\infty}\lambda^{k}\widetilde{X}^{\left(2k\right)}$,
\[
\begin{array}{cc}
n & \widetilde{X}^{\left(2n\right)}\\
1 & \frac{x^{2}}{2}\\
2 & x^{2}+\frac{x^{4}}{24}\\
3 & \frac{x^{4}}{6}+\frac{x^{6}}{720}\\
\vdots & \vdots
\end{array}
\]
Thus, the first four terms of the SPPS representation of $u_{1}$
have the form
\begin{equation}
u_{1}(x)=1+\frac{x^{2}}{2}\lambda+\left(x^{2}+\frac{x^{4}}{24}\right)\lambda^{2}+\left(\frac{x^{4}}{6}+\frac{x^{6}}{720}\right)\lambda^{3}+\cdots.\label{example u1}
\end{equation}
Notice that due to (\ref{eq:sols en x0 1}) the solution $u_{1}$
satisfies (\ref{example initial cond}) and hence must coincide with
$y$. Indeed, computation of the first four terms of the Taylor series
of the exact solution $y$ with respect to the spectral parameter
$\lambda$ gives us again (\ref{example u1}). \end{example}

\section{\label{Spectral shift} Spectral shift technique}

The procedure for constructing solutions of equation (\ref{eq: S-L Bundle})
described in Theorem \ref{thm:SL grado N} works when a particular
solution is available for $\lambda=0$ (analytically or numerically).
In \cite{SPPS} it was mentioned that for equation (\ref{eq: Sturm-Liouville})
it is also possible to construct the SPPS representation of a general
solution starting from a non-vanishing particular solution for some
$\lambda=\lambda_{0}$. Such procedure is called spectral shift and
has already proven its usefulness for numerical applications \cite{CKT Bessel Perturbed,KKB2013,SPPS}.

We show that a spectral shift technique may also be applied to equation
(\ref{eq: S-L Bundle}). Let $\lambda_{0}$ be a fixed complex number
(not necessarily an eigenvalue). Suppose that on the finite interval
$\left[a,b\right]$ a non-vanishing solution $u_{0}$ of the equation
\begin{equation}
(pu_{0}^{\prime})^{\prime}+qu_{0}=u_{0}\sum_{k=1}^{N}\lambda_{0}^{k}r_{k}\label{eq: lambda 0 particular}
\end{equation}
is known, such that the functions $u_{0}^{2}r_{k}$, $k=1,\ldots,N$
and $\frac{1}{u_{0}^{2}p}$ are continuous on $\left[a,b\right]$.
Let $\lambda=\lambda_{0}+\Lambda$ then the right hand side of (\ref{eq: S-L Bundle})
can be written in the form
\begin{align*}
u\sum_{k=1}^{N}\left(\lambda_{0}+\Lambda\right)^{k}r_{k} & =u\sum_{k=1}^{N}r_{k}\sum_{\ell=0}^{k}\binom{k}{\ell}\lambda_{0}^{\ell}\Lambda^{k-\ell}\\
 & =u\sum_{k=1}^{N}r_{k}\lambda_{0}^{k}+u\sum_{k=1}^{N}\Lambda^{k}\sum_{\ell=0}^{N-k}\binom{k+\ell}{\ell}\lambda_{0}^{\ell}r_{k+\ell}.
\end{align*}
Due to this identity equation (\ref{eq: S-L Bundle}) is transformed
again into an equation of the form (\ref{eq: S-L Bundle}):
\begin{equation}
L_{0}u=u\sum_{k=1}^{N}\Lambda^{k}\sum_{\ell=0}^{N-k}\binom{k+\ell}{\ell}\lambda_{0}^{\ell}r_{k+\ell},\label{Poly S-L lambda0}
\end{equation}
where
\[
L_{0}u=(pu^{\prime})^{\prime}+u\left(q-\sum_{k=1}^{N}r_{k}\lambda_{0}^{k}\right).
\]
Now the procedure described in Theorem \ref{thm:SL grado N} can be
applied to equation (\ref{Poly S-L lambda0}). Considering the particular
solution $u_{0}$ of (\ref{eq: lambda 0 particular}) and the functions
\[
\tilde{r}_{k}=\sum_{\ell=0}^{N-k}\binom{k+\ell}{\ell}\lambda_{0}^{\ell}r_{k+\ell}
\]
one can construct the system of recursive integrals $X^{(n)}$ and
$\widetilde{X}^{(n)}$ by applying (\ref{eq: equis tilde}) and (\ref{eq: equis sin tilde})
to functions $u_{0}$ and $\tilde{r}_{k}$. Then the general solution
of equation (\ref{eq: S-L Bundle}) has the form $u=c_{1}u_{1}+c_{2}u_{2}$
where $c_{1}$ and $c_{2}$ are arbitrary complex constants and
\[
u_{1}=u_{0}\sum_{n=0}^{\infty}\left(\lambda-\lambda_{0}\right)^{n}\widetilde{X}^{\left(2n\right)}\quad \text{and}\quad u_{2}=u_{0}\sum_{n=0}^{\infty}\left(\lambda-\lambda_{0}\right)^{n}X^{\left(2n+1\right)}.
\]

\section{The generalized Zakharov-Shabat system}

The Zakharov-Shabat system arises in the Inverse Scattering Transform
method when integrating the non-linear Schrödinger equation, see,
e.g., \cite{Ablowitz y Segur,Desaix,Z-S Dirac,KiraRosu2010,Kravchenko-Velasco,Tsoy,Yang2010,Zakharov-Shabat}.
In the recent work \cite{Kravchenko-Velasco} an SPPS representation
was obtained for the solutions of the Zakharov-Shabat system with
a real-valued potential. In this section we consider the generalized
Zakharov-Shabat system, also sometimes called the one-dimensional
Dirac system \cite{KiraRosu2010}
\begin{align}
v_{1}^{\prime} & =\lambda v_{1}+Pv_{2}\label{eq:ZS1}\\
v_{2}^{\prime} & =-\lambda v_{2}-Qv_{1},\label{eq:ZS2}
\end{align}
where $v_{1}$ and $v_{2}$ are unknown complex valued functions of
the independent variable $x$, $\lambda\in\mathbb{C}$ is a constant,
$Q$ and $P$ are complex-valued functions such that $Q$ does not
vanish on the domain of interest.

From (\ref{eq:ZS2}) we have
\begin{equation}
v_{1}=-\frac{1}{Q}\left(v_{2}^{\prime}+\lambda v_{2}\right).\label{eq: v1 despejada}
\end{equation}
Substituting this expression into (\ref{eq:ZS1}) we obtain an equation of
the form (\ref{eq: S-L Bundle}):
\begin{equation}
\left(\frac{1}{Q}v_{2}^{\prime}\right)^{\prime}+Pv_{2}=\lambda\frac{Q^{\prime}}{Q^{2}}v_{2}+\lambda^{2}\frac{1}{Q}v_{2}.\label{eq:ZS-pencil}
\end{equation}

\subsection{\label{sub:Solution-for-the}SPPS for the generalized Zakharov-Shabat
system}

In this subsection, applying the SPPS method to equation (\ref{eq:ZS-pencil}) we obtain an SPPS representation for solutions of the system \eqref{eq:ZS1}, \eqref{eq:ZS2}.  The following statement is a direct application of Theorem \ref{thm:SL grado N}
to equation (\ref{eq:ZS-pencil}).

\begin{corollary} \label{thm:SPPS pencil} Suppose that on a finite
segment $\left[a,b\right]$ the equation
\begin{equation}
\left(\frac{1}{Q}v^{\prime}\right)^{\prime}+Pv=0.\label{eq:ZS-pencil-cero}
\end{equation}
possesses a non vanishing particular solution $v_{0}\in C^{1}[a,b]$ such that
$\frac{v_{0}^{\prime}}{Q}\in C^{1}[a,b]$ and the functions $v_{0}^{2}\frac{Q^{\prime}}{Q^{2}}$,
$\frac{v_{0}^{2}}{Q}$ and $\frac{Q}{v_{0}^{2}}$ are continuous on
$\left[a,b\right]$. Then the general solution of \eqref{eq:ZS-pencil}
has the form $v_{2}=c_{1}g_{1}+c_{2}g_{2}$, where $c_{1}$ and $c_{2}$
are arbitrary complex constants,
\begin{equation}
g_{1}=v_{0}\sum_{k=0}^{\infty}\lambda^{k}\widetilde{X}^{\left(2k\right)}\quad \text{and}\quad g_{2}=v_{0}\sum_{k=0}^{\infty}\lambda^{k}X^{\left(2k+1\right)}\label{eq:Sumas-g}
\end{equation}
with $\widetilde{X}^{\left(n\right)}$ and $X^{\left(n\right)}$ being
defined by the recursive relations
\begin{align*}
\widetilde{X}^{\left(n\right)} & \equiv X^{\left(n\right)}\equiv0\qquad\text{for }n<0,\\
\widetilde{X}^{\left(0\right)} & \equiv X^{\left(0\right)}\equiv1,
\end{align*}
and
\begin{align}
\widetilde{X}^{\left(n\right)}\left(x\right) & =\begin{cases}
\intop_{x_{0}}^{x}\left(\widetilde{X}^{\left(n-1\right)}\left(y\right)v_{0}^{2}\left(y\right)\frac{Q^{\prime}\left(y\right)}{Q^{2}\left(y\right)}+\widetilde{X}^{\left(n-3\right)}\left(y\right)\frac{v_{0}^{2}\left(y\right)}{Q\left(y\right)}\right)dy, & n\text{-odd}\\
\intop_{x_{0}}^{x}\widetilde{X}^{\left(n-1\right)}\left(y\right)\frac{Q\left(y\right)}{v_{0}^{2}\left(y\right)}dy, & n\text{-even}
\end{cases}\label{eq: X tilde ZS}\\
X^{\left(n\right)}\left(x\right) & =\begin{cases}
\intop_{x_{0}}^{x}X^{\left(n-1\right)}\left(y\right)\frac{Q\left(y\right)}{v_{0}^{2}\left(y\right)}dy, & n\text{-odd}\\
\intop_{x_{0}}^{x}\left(X^{\left(n-1\right)}\left(y\right)v_{0}^{2}\left(y\right)\frac{Q^{\prime}\left(y\right)}{Q^{2}\left(y\right)}+X^{\left(n-3\right)}\left(y\right)\frac{v_{0}^{2}\left(y\right)}{Q\left(y\right)}\right)dy, & n\text{-even}
\end{cases}\label{eq: X ZS}
\end{align}
where $x_{0}$ is an arbitrary point of $\left[a,b\right]$. Furthermore,
both series in \eqref{eq:Sumas-g} converge uniformly on $\left[a,b\right]$.
\end{corollary}

Corollary \ref{thm:SPPS pencil} allows us to construct a general
solution of the Zakharov-Shabat system.

\begin{theorem} \label{thm:Sol-ZS} Under the conditions of Corollary \ref{thm:SPPS pencil} the general solution of
the system \eqref{eq:ZS1}, \eqref{eq:ZS2} has the form
\begin{align}
v_{1} & =-c_{1}\left(\frac{v_{0}^{\prime}+\lambda v_{0}}{Q}\sum_{k=0}^{\infty}\lambda^{k}\widetilde{X}^{\left(2k\right)}+\frac{\lambda}{v_{0}}\sum_{k=0}^{\infty}\lambda^{k}\widetilde{X}^{\left(2k+1\right)}\right)\label{eq: v1 ZS}\\
 & \qquad-c_{2}\left(\frac{v_{0}^{\prime}+\lambda v_{0}}{Q}\sum_{k=0}^{\infty}\lambda^{k}X^{\left(2k+1\right)}+\frac{1}{v_{0}}\sum_{k=0}^{\infty}\lambda^{k}X^{\left(2k\right)}\right),\nonumber \\
v_{2} & =c_{1}v_{0}\sum_{k=0}^{\infty}\lambda^{k}\widetilde{X}^{\left(2k\right)}+c_{2}v_{0}\sum_{k=0}^{\infty}\lambda^{k}X^{\left(2k+1\right)}\label{eq:v2}
\end{align}
where $c_{1}$ and $c_{2}$ are arbitrary complex constants and the
functions $X^{(n)}$ and $\widetilde{X}^{(n)}$ are the same as in
Corollary \ref{thm:SPPS pencil}. \end{theorem}

\begin{proof} By Corollary \ref{thm:SPPS pencil} a general solution of \eqref{eq:ZS-pencil} is given by (\ref{eq:v2}).
Define $v_{1}$ by (\ref{eq: v1 despejada}), then $v_{1}$ has the
form (\ref{eq: v1 ZS}) and the pair $v_{1}$, $v_{2}$ satisfy \eqref{eq:ZS1}, \eqref{eq:ZS2}. \end{proof}

\begin{remark} \label{Rem Comparison}In comparison with \cite{Kravchenko-Velasco}
where the additional conditions $Q=P$ and $Q$ being real-valued
were required, this theorem establishes the SPPS representation for
solutions of the system \eqref{eq:ZS1}, \eqref{eq:ZS2} allowing complex-valued coefficients, additionally requiring one of them being non-vanishing. \end{remark}

\subsection{The Zakharov-Shabat eigenvalue problem\label{ZShabat eigenvalue problem}}

In this subsection we consider classical Zakha\-rov-Shabat systems characterized
by the condition $Q=P^{\ast}$ (where $\ast$ denotes the complex
conjugation). This case arises in physical models associated with
optical solitons, see, e.g., \cite{Ablowitz y Segur,Desaix,Klaws Shaw Elsevier,Klaus Shaw,Kravchenko-Velasco,Zakharov-Shabat}.

\begin{definition}[\cite{Klaws Shaw Elsevier,Klaus Shaw}]
Solutions
of the Zakharov-Shabat system (\ref{eq:ZS1}), (\ref{eq:ZS2}) satisfying
the following asymptotic relations:
\begin{align*}
\overrightarrow{\sigma} & \cong\binom{1}{0}e^{\lambda x},\qquad x\rightarrow-\infty\\
\overrightarrow{\xi} & \cong\binom{0}{1}e^{-\lambda x},\qquad x\rightarrow+\infty
\end{align*}
for some $\lambda$ with $\operatorname{Re}\lambda>0$, are called
Jost solutions. Expression $\binom{\sigma_{1}}{\sigma_{2}}\cong\binom{1}{0}e^{\lambda x},$
$x\rightarrow\pm\infty$ means that there exists $\lim_{x\rightarrow\pm\infty}\frac{\sigma_{1}(x)}{e^{\lambda x}}=c$,
$c\neq0$, while $\sigma_{2}(x)=o\left(\sigma_{1}(x)\right)$, $x\rightarrow\pm\infty$.
\end{definition}

The eigenvalue problem for the Zakharov-Shabat system consists in
finding such values of the spectral parameter $\lambda$ for which
there exists a nontrivial Jost solution.

In particular, when the potential $Q$ is compactly supported on $\left[-a,a\right]$ (this situation usually appears after truncating an infinitely supported potential),
it is easy to see that the eigenvalue problem reduces to finding such
values of $\lambda$ ($\operatorname{Re} \lambda>0$) for which there exists a solution
of (\ref{eq:ZS1}), (\ref{eq:ZS2}) on $(-a,a)$ satisfying the following
boundary conditions (see, e.g., \cite{Klaus Shaw}, \cite{Klaus Shaw ZS}, \cite{Kravchenko-Velasco})
\begin{align}
v_{1}\left(-a\right)=1,\quad & v_{2}\left(-a\right)=0,\label{eq: Jost -a}\\
v_{1}\left(a\right) & =0.\label{eq: Jost +a}
\end{align}
We assume additionally that the potential $Q$ does not vanish on
its support.

\subsection{Dispersion equation for the eigenvalue problem}\label{Subsect DE ZS}

In this subsection we write down the dispersion (or characteristic)
equation equivalent to the Zakharov-Shabat eigenvalue problem with
a compactly supported potential.

\begin{theorem} Let $Q$ be a continuous complex-valued non-vanishing
function on $\left[-a,a\right]$, and $v_{0}$ be a particular non-vanishing
solution of \eqref{eq:ZS-pencil-cero} satisfying the conditions of Corollary \ref{thm:SPPS pencil}. Then $\lambda$ ($\operatorname{Re}\lambda>0$)
is an eigenvalue of the spectral problem for the Zakharov-Shabat system
\eqref{eq:ZS1}, \eqref{eq:ZS2}, \eqref{eq: Jost -a}, \eqref{eq: Jost +a}
if and only if the following equation is satisfied
\begin{equation}
\sum_{k=0}^{\infty}\lambda^{k}\left(v_{0}\left(a\right)\left(v_{0}^{\prime}\left(a\right)X^{\left(2k+1\right)}(a)+v_{0}\left(a\right)X^{\left(2k-1\right)}\left(a\right)\right)+Q(a)X^{\left(2k\right)}\left(a\right)\right)=0,\label{eq: caracteristica}
\end{equation}
where the functions $X^{(n)}$ are defined by \eqref{eq: X ZS} with
$x_{0}=-a$. \end{theorem}

\begin{proof} Considering $v_{1}$ and $v_{2}$ from Theorem \ref{thm:Sol-ZS}
with $x_{0}=-a$ we have $v_{1}\left(-a\right)=-c_{1}\left(\frac{v_{0}^{\prime}+\lambda v_{0}\left(-a\right)}{Q\left(-a\right)}\right)-c_{2}\left(\frac{1}{v_{0}\left(-a\right)}\right)$
and $v_{2}\left(-a\right)=c_{1}v_{0}(-a)$. From the boundary condition
(\ref{eq: Jost -a}) we obtain that $c_{1}=0$ and $c_{2}=-v_{0}\left(-a\right)$.
According to the boundary condition (\ref{eq: Jost +a}) we obtain
that the characteristic equation has the form
\[
\frac{v_{0}^{\prime}\left(a\right)+\lambda v_{0}\left(a\right)}{Q\left(a\right)}\sum_{k=0}^{\infty}\lambda^{k}X^{\left(2k+1\right)}\left(a\right)+\frac{1}{v_{0}\left(a\right)}\sum_{k=0}^{\infty}\lambda^{k}X^{\left(2k\right)}\left(a\right)=0
\]
which is equivalent to (\ref{eq: caracteristica}) taking into account
that $X^{(-1)}\equiv0$. \end{proof}

\begin{remark} When $Q$ is real valued, the characteristic equation
from \cite{Kravchenko-Velasco} can be used since it was obtained
without requiring that $Q$ should be non-vanishing. \end{remark}

\begin{remark}
It is possible to apply the spectral shift technique described in Section \ref{Spectral shift} to the solution of the Zakharov-Shabat system. Note that we change $\lambda$ by $\lambda-\lambda_0$ only in \eqref{eq:Sumas-g} and keep the parameter $\lambda$ in \eqref{eq: v1 despejada}. Hence the characteristic equation \eqref{eq: caracteristica} under the spectral shift becomes
\begin{equation}
\sum_{k=0}^{\infty}(\lambda-\lambda_{0})^{k}\left(v_{0}\left(a\right)
\left((v_{0}^{\prime}\left(a\right)+\lambda_{0}v_{0}\left(a\right))
X^{\left(2k+1\right)}(a)+v_{0}\left(a\right)X^{\left(2k-1\right)}
\left(a\right)\right)+Q(a)X^{\left(2k\right)}\left(a\right)\right)=0.\label{eq:caracteristicashft}
\end{equation}
\end{remark}

This theorem reduces the Zakharov-Shabat eigenvalue problem with a
compactly supported potential to the problem of localizing zeros (in
the right half-plane) of an analytic function $\Phi(\lambda)=\sum_{k=0}^{\infty}a_{k}\lambda^{k}$
of the complex variable $\lambda$ with the Taylor coefficients $a_{k}$
given by the expressions,
\[
a_{k}=v_{0}\left(a\right)\left(v_{0}^{\prime}\left(a\right)X^{\left(2k+1\right)}+v_{0}\left(a\right)X^{\left(2k-1\right)}\left(a\right)\right)+Q(a)X^{\left(2k\right)}\left(a\right).
\]

The coefficients $a_{k}$ can be easily and accurately calculated
following the definitions introduced above. For the numerical solution
of the eigenvalue problem one can truncate the series (\ref{eq: caracteristica})
and consider a polynomial
\begin{equation}
\Phi_{M}\left(\lambda\right)=\sum_{k=0}^{M}a_{k}\lambda^{k}\label{Poly Z-S}
\end{equation}
approximating the function $\Phi$. For a reasonably large $M$ some
of its roots give an accurate approximation of the eigenvalues of
the problem. The Rouché theorem establishes, see, e.g., \cite[Section 3]{Conway},
that if the complex-valued functions $f$ and $g$ are holomorphic
inside and on some closed simple contour $K$, with $|g(z)|<|f(z)|$
on $K$, then $f$ and $f+g$ have the same number of zeros inside
$K$, where each zero is counted as many times as its multiplicity.
As it follows from the Rouché theorem, the roots of $\Phi_{M}$ closest
to zero give an accurate approximation of the eigenvalue problem whilst
the roots more distant from the origin are spurious roots appearing
due to the truncation. Indeed, consider a domain $\Omega$ in the
complex plane of the variable $\lambda$ such that
\begin{equation}
\min_{\lambda\in\partial\Omega}\left\vert \Phi_{M}\left(\lambda\right)\right\vert >\max_{\lambda\in\partial\Omega}\left\vert \Phi(\lambda)-\Phi_{M}\left(\lambda\right)\right\vert \label{inequality from Rouche}
\end{equation}
(that is, $f=\Phi_{M}$, $g=\Phi-\Phi_{M}$ and hence $f+g=\Phi$).
Then the number of zeros of $\Phi_{M}$ in $\Omega$ coincides with
the number of zeros of $\Phi$, or which is the same with the number
of eigenvalues located in $\Omega$. In other words, in a domain $\Omega$
where $\Phi_{M}$ approximates sufficiently closely the function $\Phi$
(the inequality (\ref{inequality from Rouche}) should be fulfilled)
all the roots of $\Phi_{M}$ approximate the true eigenvalues of the
spectral problem. In order to estimate the quantity $\max_{\lambda\in\partial\Omega}\left\vert \Phi(\lambda)-\Phi_{M}\left(\lambda\right)\right\vert $
the estimates derived in the proof of Theorem \ref{thm:SL grado N}
can be used. Obviously, the same reasoning is applicable to the characteristic function from \eqref{eq:caracteristicashft} centered in some $\lambda_0$.

The following example illustrates the application of the SPPS method
to a Zakharov-Shabat spectral problem admitting complex eigenvalues.

\begin{example} \label{Example Z-S Potential Q}Consider the Zakharov-Shabat
system with the potential
\begin{equation}
Q(x)=\begin{cases}
s\left(-1+3\frac{\pi}{4}+3x^{2}\right), & -1\leq x\leq1,\\
0, & \text{ otherwise}
\end{cases}\label{eq: Potencial polinomio}
\end{equation}
where $s\in\mathbb{R}$. This potential was considered in a numerical
experiment in \cite{Klaus Shaw}.

According to \cite{Klaus Shaw}, there is a pair of complex eigenvalues
located symmetrically about the real axis in the right half-plane
when $s$ is in the range $0.956\le s\le 0.9999$. As $s$ increases
the eigenvalues $\lambda_{1}$ and $\lambda_{2}$ approach each other and eventually coalesce into a double eigenvalue. If we further increase the parameter
$s$ a pair of real eigenvalues appear after the passage through a
double eigenvalue state. Table \ref{eq: Tabla lambda 1,2} computed
by means of the SPPS method illustrates the described phenomenon. With the help of the SPPS representation we found more precisely the value of the parameter $s$ when the eigenvalues $\lambda_{1}$ and $\lambda_{2}$ coalesce into a double eigenvalue. The value $s\approx0.9999006472847$ corresponds to the moment when the eigenvalues are the closest.

For the numerical computation we used MATLAB and approximated (\ref{eq: caracteristica})
with $M=100$ and $v_{0}=\exp\left(i\int Q\right)$ as a particular
solution. The recursive integrals $X^{(n)}$ were calculated using
the Newton-Cottes 6 point integration formula of 7-th order (see,
e.g., \cite{Rabinowitz}).

\begin{table}[tbh]
\centering\small
\begin{tabular}{cll}
\hline
$s$  & \multicolumn{1}{c}{$\lambda_{1}$} & \multicolumn{1}{c}{$\lambda_{2}$}\tabularnewline
\hline
$0.956$  & $0.0000544585364-0.6265762379200i$  &
           $0.0000544585364+0.6265762379200i$\tabularnewline
$0.967$  & $0.0076637690047-0.5443495752993i$  &
           $0.0076637690047+0.5443495752993i$\tabularnewline
$0.989$  & $0.0227377545015-0.3155449553793i$  &
           $0.0227377545015+0.3155449553793i$\tabularnewline
$0.9999$ & $0.0301375300344-0.0027328986939i$  &
           $0.0301375300365+0.0027328986939i$\tabularnewline
$0.9999006472847$ & $0.0301625594632+5.0785\cdot10^{-9}i$ &
                    $0.0301635517506-5.0785\cdot10^{-9}i$\tabularnewline
$0.999901$& $0.0283635450766+2.4\cdot10^{-12}i$  &
            $0.0319628654396-2.4\cdot10^{-12}i$\tabularnewline
$0.99991$ & $0.0209514552194+4.9\cdot10^{-13}i$  &
            $0.0393371418191-4.9\cdot10^{-13}i$\tabularnewline
$0.99995$ & $0.0089010060464+1.3\cdot10^{-13}i$  &
            $0.0514417381256-1.3\cdot10^{-13}i$\tabularnewline
$0.99999$ & $0.0015563906608+2.8\cdot10^{-14}i$  &
            $0.0588404994651-3.1\cdot10^{-14}i$\tabularnewline
\hline
\end{tabular}\caption{Eigenvalues behavior depending on the parameter $s$ in Example \ref{Example Z-S Potential Q}.}
\label{eq: Tabla lambda 1,2}
\end{table}

The SPPS method allows one to obtain a graph of the characteristic
function in 3D within several seconds. The location of the roots
on the complex plane for different values of the parameter $s$ is
illustrated by Figure \ref{fig: char eq 3d} where the graphs of the
function $-\log\left\vert \Phi_{M}\right\vert $ are plotted.

\begin{figure}[tbh]
\centering
\begin{tabular}{cc}
\includegraphics[
natwidth=900,
natheight=720,
width=3in,
height=2.4in]{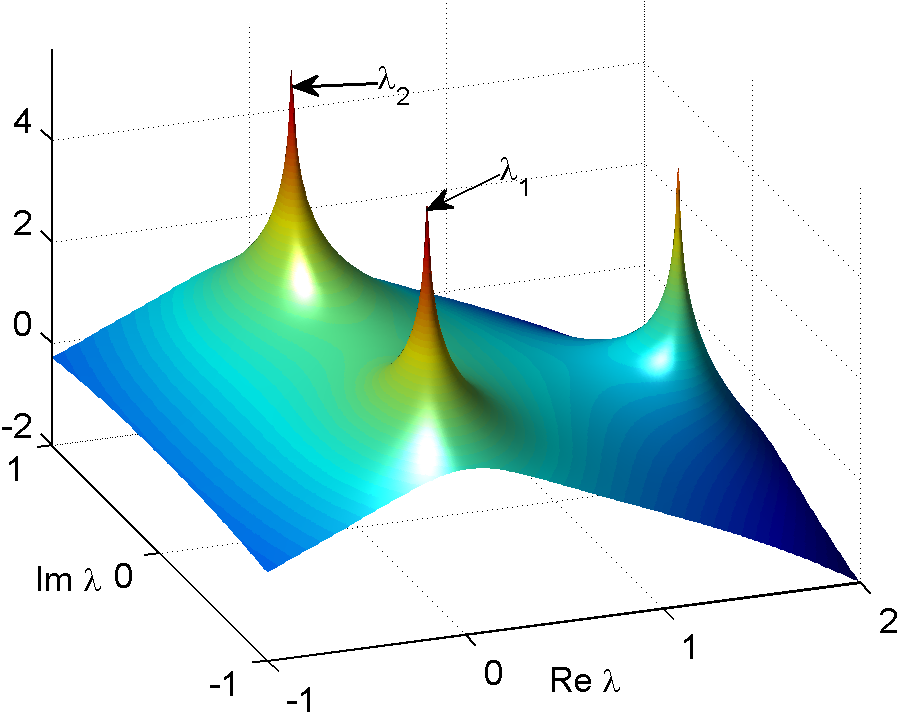} &
\includegraphics[
natwidth=900,
natheight=720,
width=3in,
height=2.4in]{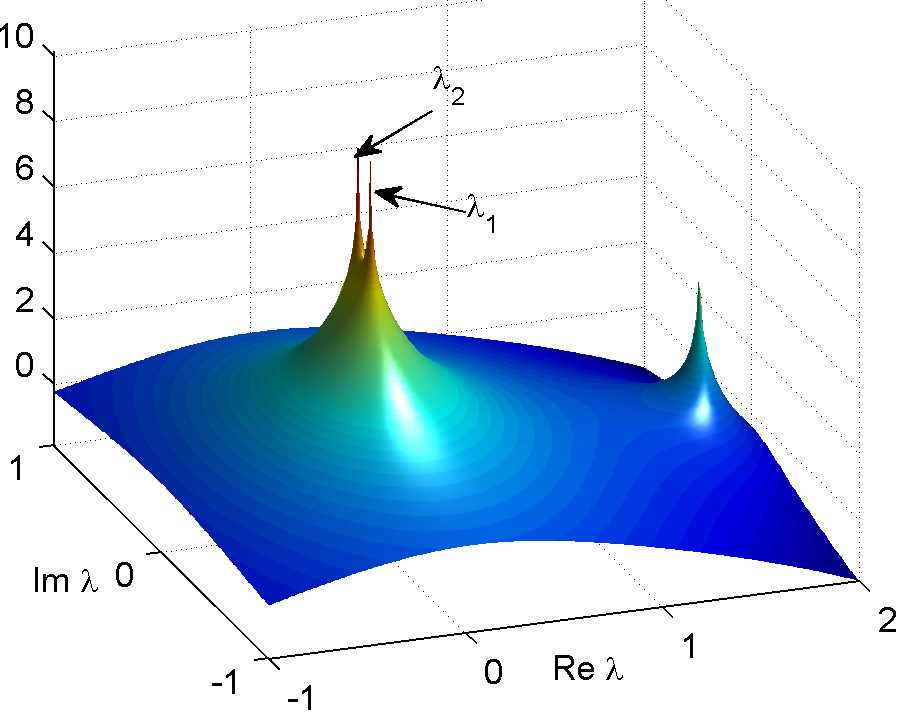} \\
(a) & (b) \medskip\\
\includegraphics[
natwidth=900,
natheight=720,
width=3in,
height=2.4in]{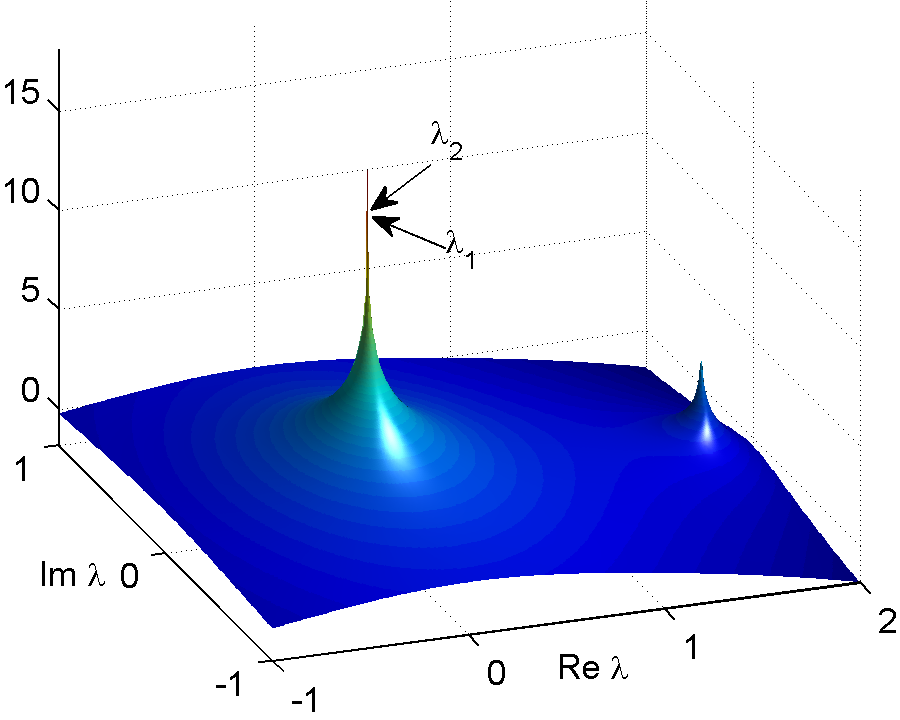} &
\includegraphics[
natwidth=900,
natheight=720,
width=3in,
height=2.4in]{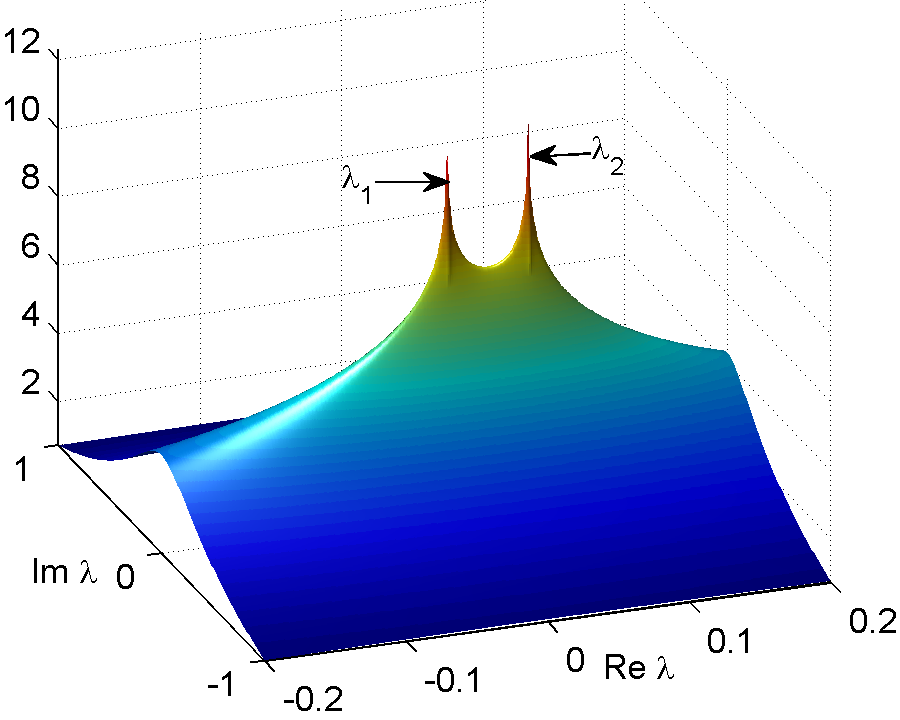} \\
(c) & (d)
\end{tabular}
\caption{Graphs of the function $-\log\left\vert \Phi_{M}\right\vert $ from
Example \ref{Example Z-S Potential Q}. On (a) the peaks on the left
show the position of $\lambda_{1}$ and $\lambda_{2}$ when $s=0.956$.
On (b) at $s=0.9995$ it is possible to see that the peaks on the
left are closer to each other. On (c) the peaks of $\lambda_{1}$
and $\lambda_{2}$ are near to coalesce into a double eigenvalue at
$s=0.9999006472847$, and finally on (d) $\lambda_{1}$ and $\lambda_{2}$
are two different real eigenvalues at $s=0.99999$.}
\label{fig: char eq 3d}
\end{figure}
\end{example}

\subsection{The argument principle}

Since the left-hand side of equation (\ref{eq: caracteristica}) is
an analytic function with respect to $\lambda$, it is possible to
apply the argument principle to locate its zeros and compute their
order. Of course, in the case when an approximate characteristic function
is just a polynomial as in \eqref{Poly Z-S}, various methods
of accurate localization of zeros are known and the usage of the argument
principle can be excessive. Nevertheless even in this case the argument
principle can be useful. As we already know only several roots of
the polynomial \eqref{Poly Z-S} are approximations of the eigenvalues,
all other roots appear due to the truncation of the
series \eqref{eq: caracteristica}. The localization of several roots
closest to the origin using the argument principle followed by several
Newton iterations can be faster than the localization of all the roots
by some general-purpose method. Moreover for more general boundary
conditions, e.g., involving multiplication by an analytic function of $\lambda$, the approximate characteristic function may not
be a polynomial, and the argument principle becomes even more useful (see,
e.g., \cite{Bronski,Dellnitz Schutze Zheng,Ying Katz}). The argument principle consists in the following, see, e.g., \cite{Conway}.

\begin{theorem}[The argument principle]Let $f$ be an analytic function
in a domain $G$ with zeros $z_{1}$, $z_{2}$,...,$z_{n}$ counted
according to their multiplicity. If $\gamma$ is a simple closed rectifiable
curve in $G$, contractible to a point in $G$ and not passing through $z_{1}$, $z_{2}$,...,$z_{n}$,
then
\begin{equation}
\frac{1}{2\pi i}\ointop_{\gamma}\frac{f^{\prime}(z)}{f(z)}dz=\sum_{k=1}^{n}n(\gamma;z_{k}),\label{eq: Integral princ argum}
\end{equation}
where $n(\gamma; z_k)$ is the winding number of $\gamma$ around $z_k$.
\end{theorem}

This theorem applied to the approximate characteristic function (\ref{Poly Z-S})
in a domain restricted by the Rouché theorem (see the discussion above)
allows one to find a precise number of eigenvalues
and can even be used for their accurate location. We illustrate this approach
in the next example for which a routine in MATLAB was written locating
positions of zeros of the approximate characteristic function (\ref{Poly Z-S})
by computing the argument change using $4000$ points along rectangular
contours $\gamma$. If the argument change along certain $\gamma$
is zero the program considers another rectangular contour. Otherwise
the program divides the rectangle into two smaller rectangles and calls the
routine again on each rectangle until a desired tolerance is achieved.
The program evaluates (\ref{eq: Integral princ argum}) over the final
contour $\gamma$ to find the order of the zero and also refines the
position of the zero by the well known residue formula (see e.g. \cite{Bronski,Kravanja})
\[
z_k=\frac{1}{2\pi i N(z_k)}\ointop_{\gamma}\frac{z\Phi_{M}^{\prime}\left(z\right)}{\Phi_{M}\left(z\right)}dz,
\]
where $N(z_k)$ is the multiplicity of the zero $z_k$.

\begin{example} \label{Example Z-S Potential Q ArgPrinciple} Consider
the Zakharov-Shabat system with the potential (\ref{eq: Potencial polinomio}).
We compute the polynomial (\ref{Poly Z-S}) with $M=100$ and present
on Figure \ref{fig: arg principle} the illustration of the work of
the described algorithm. The shadowed areas contain the corresponding
eigenvalues. Eigenvalues $\lambda_{1}$ and $\lambda_{2}$ get closer
to each other and then separate. The plots illustrate the convergence
of the described procedure to the approximate eigenvalues presented
in Table \ref{eq: Tabla Arg Principle}.

\begin{figure}[tbh]
\centering
\begin{tabular}{cc}
\includegraphics[bb=212 324 399 468,width=2.6in,height=2in]{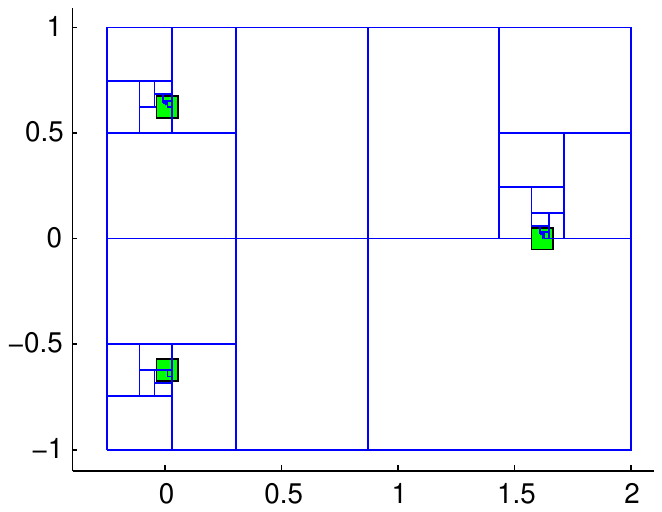}\ & \ \includegraphics[bb=212 324 399 468,width=2.6in,height=2in]{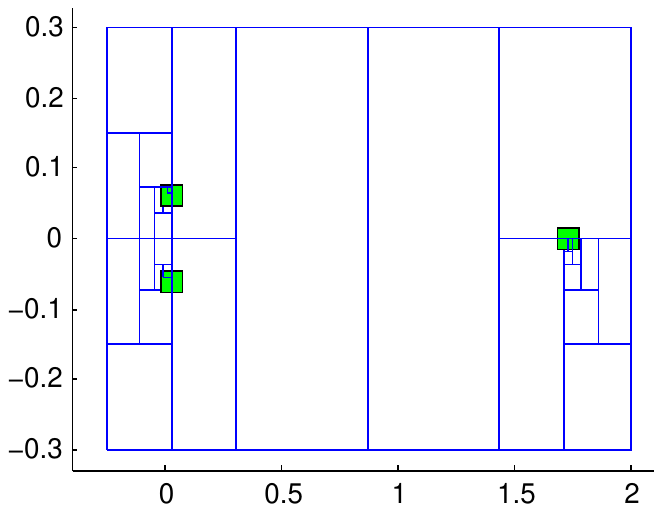}\\
~\\
\includegraphics[bb=212 324 399 468,width=2.6in,height=2in]{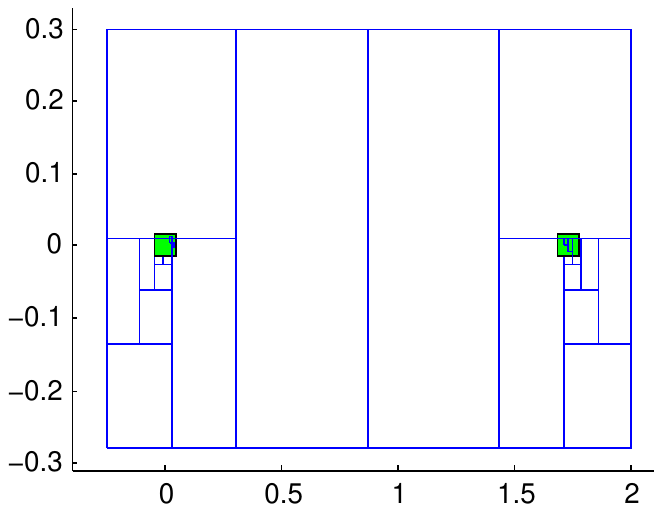}\ &
\ \includegraphics[bb=212 324 399 468,width=2.6in,height=2in]{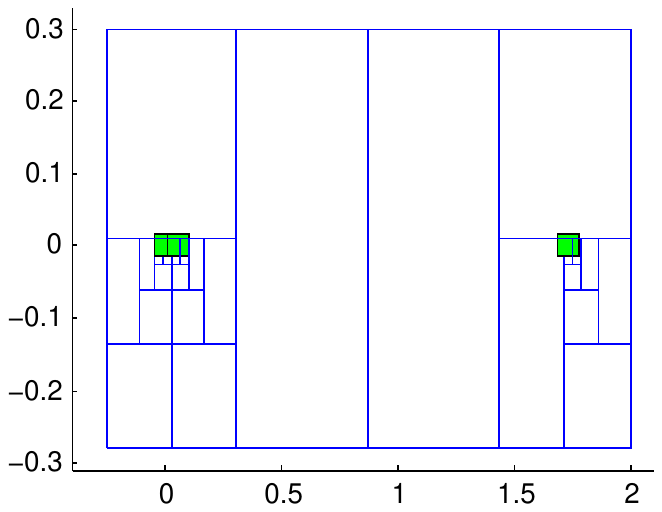}
\end{tabular}
\caption{Illustration of the algorithm based on the argument principle in Example
\ref{Example Z-S Potential Q ArgPrinciple} for $s=0.956$ (top left),
$s=0.9995$ (top right), $s=0.9999006472847$ (bottom left) and $s=0.99999$
(bottom right). Shadowed rectangles mark positions of the eigenvalues.}
\label{fig: arg principle}
\end{figure}

\begin{table}[tbh]
\centering %
\begin{tabular}{ccc}
\hline
$s$  & $\lambda_{1,2}$  & $\lambda_{3}$  \tabularnewline
\hline
$0.956$  & $0.0124269786\pm0.6259290726i$  & $1.6264287578+2\cdot10^{-15}i$  \tabularnewline
$0.9995$ & $0.0299931318\pm0.0606462518i$  & $1.7346285373-1.5\cdot10^{-14}i$  \tabularnewline
$0.9999006472847$  & %
\begin{tabular}{c}
$0.0301630556+3.9\cdot10^{-15}i$,\tabularnewline
double eigenvalue\tabularnewline
\end{tabular} & $1.7356369172-2\cdot10^{-15}i$  \tabularnewline
$0.99999$  & %
\begin{tabular}{l}
$0.0588404994+3.1\cdot10^{-14}i$,\tabularnewline
$0.0015563906-2.8\cdot10^{-14}i$\tabularnewline
\end{tabular} & $1.7358620206+10^{-15}i.$  \tabularnewline
\hline
\end{tabular}\caption{Eigenvalues behavior depending on the parameter $s$ in Example \ref{Example Z-S Potential Q}.}
\label{eq: Tabla Arg Principle}
\end{table}
\end{example}

\subsection{Zakharov-Shabat systems with complex potentials.}

In this subsection we compare the performance of the SPPS method with
the results presented in \cite{Bronski,Tovbis} where the semi-classical
scaling of the non-self-adjoint Zakharov-Shabat scattering problem
\begin{eqnarray}
i\epsilon v_{x} & = & q\omega+\Lambda\upsilon\label{eq:nsazs1}\\
i\epsilon\omega_{x} & = & q^{*}\upsilon-\Lambda\omega,\label{eq:nsazs2}
\end{eqnarray}
is considered, see also \cite{KLL2013}. The potential function $q$ is of the form
\begin{eqnarray}
q(x) & = & A(x)e^{iS(x)/\epsilon}.\label{eq: potencial q}
\end{eqnarray}
Considering notations $Q=\frac{i}{\epsilon}q^{*}$ and $\lambda=-\frac{i}{\epsilon}\Lambda$ one obtains
the Zakharov-Shabat system (\ref{eq:ZS1}), (\ref{eq:ZS2}).

In order to apply the results of Subsections \ref{ZShabat eigenvalue problem} and \ref{Subsect DE ZS} we approximate the potential $q$ by a function compactly supported on $[-a,a]$. If $a$ is chosen to be sufficiently
large the error due to this truncation of the potential is sufficiently small.

\begin{example}[\cite{Bronski}] \label{Example String 1-1}Consider the functions
\[
A(x) = \sech(2x),\qquad  S(x) = \sech(2x)
\]
for the potential \eqref{eq: potencial q}.
We approximated the potential with
\begin{equation*}
\hat{q}(x)=\begin{cases}
A(x)e^{iS(x)/\epsilon}, & -a\leq x\leq a,\\
0, & \text{ otherwise}
\end{cases}
\end{equation*}
for $a=10$. The approximate eigenvalues obtained by the SPPS method for various values of $\epsilon$ are shown on Figure \ref{Fig1-1} and agree with the results from \cite{Bronski}. The presented results were obtained with the help of the spectral shift technique in Matlab using machine precision arithmetics, $M=250$ and 100000 points for the computation of the formal powers. For the smaller
values of $\epsilon$ considered in \cite{Bronski} the SPPS method
required to use high-precision arithmetics due to high oscillatory
potential $q$ and nearly vanishing particular solution $u_{0}$.
We chose not to include these numerical experiments.

\begin{figure}[tbh]
\centering
\begin{tabular}{ccc}
\includegraphics[
bb=219 324 391 468,
width=2.4in,
height=2in]{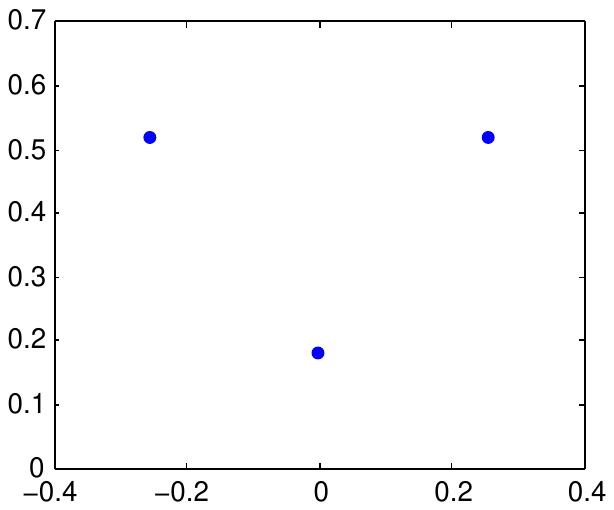} & &
\includegraphics[
bb=219 324 391 468,
width=2.4in,
height=2in]{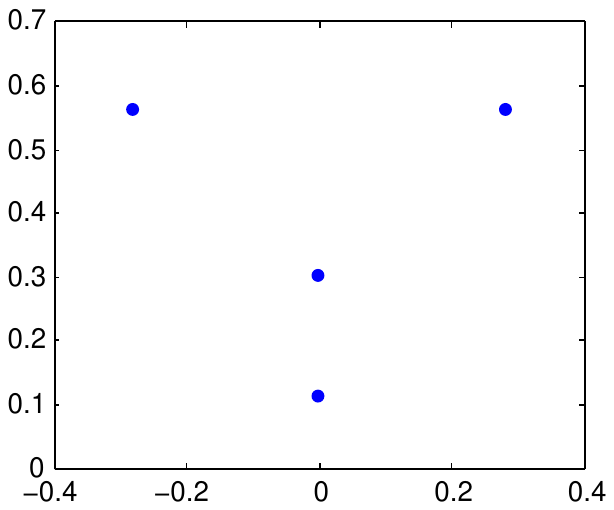}\\
(a) $\epsilon$=0.2 & & (b) $\epsilon$=0.159\\
\includegraphics[
bb=219 324 391 468,
width=2.4in,
height=2in]{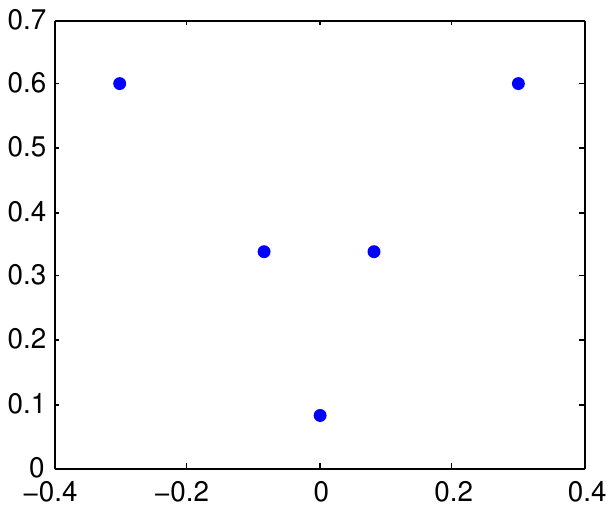} & &
\includegraphics[
bb=219 324 391 468,
width=2.4in,
height=2in]{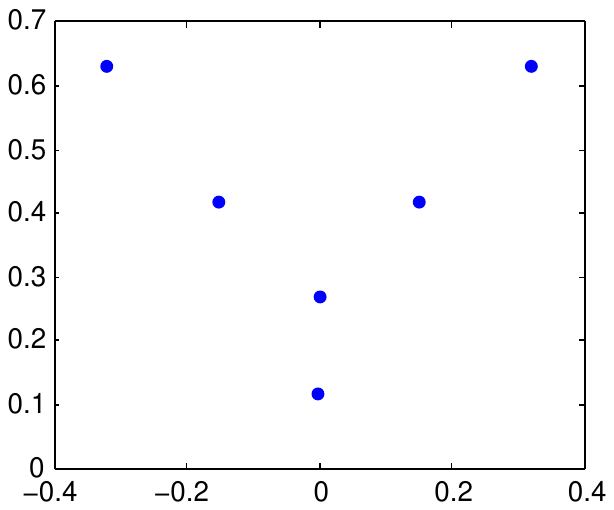}\\
(c) $\epsilon$=0.126 & & (d) $\epsilon$=0.1 \\
\includegraphics[
bb=219 324 391 468,
width=2.4in,
height=2in]{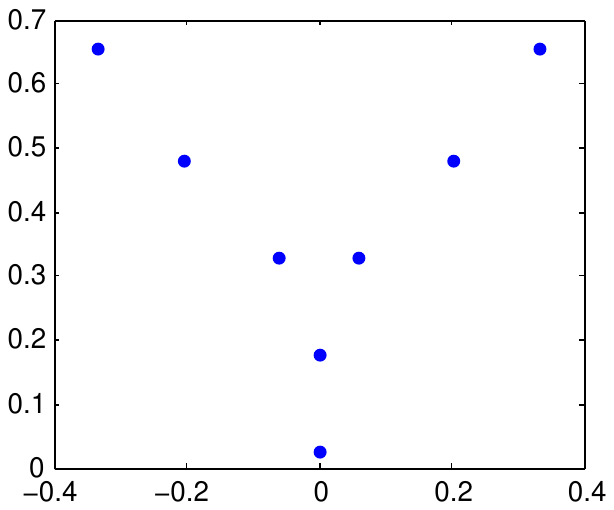} & &
\includegraphics[
bb=219 324 391 468,
width=2.4in,
height=2in]{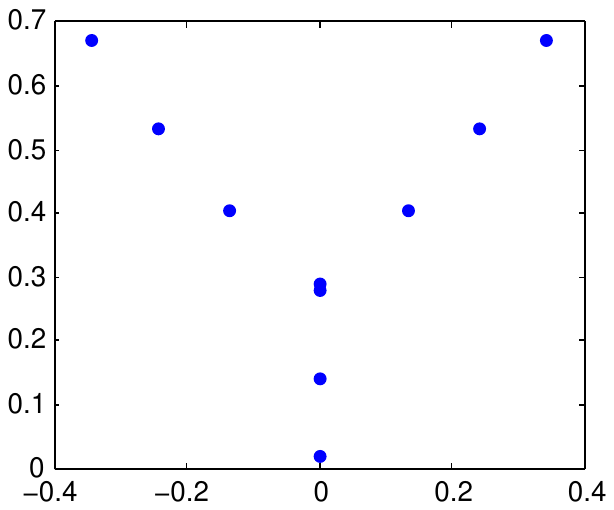}\\
(e) $\epsilon$=0.0794 & & (f) $\epsilon$=0.063
\end{tabular}
\caption{Eigenvalue locations from Example \ref{Example String 1-1} for $\epsilon$=0.2,
0.159, 0.126, 0.1, 0.0794 and 0.063.}
\label{Fig1-1}
\end{figure}
\end{example}

\begin{example} \label{Example String 1-1-1}
In \cite{Tovbis} the Zakharov-Shabat system (\ref{eq:nsazs1}), (\ref{eq:nsazs2})
with the potential \eqref{eq: potencial q} was considered for
\[
A(x)=-\sech x,\qquad S^{\prime}(x)=-\mu\tanh x.
\]
This potential is interesting because an exact formula for its
eigenvalues is known and is given by
\[
z_{n}=i\left(\sqrt{1-\frac{1}{4}\mu^{2}}-\varepsilon\left(n-\frac{1}{2}\right)\right),
\]
where the index $n$ ranges strictly over the positive integers that
satisfy
\begin{equation}
n(\mu,\varepsilon)<\frac{1}{\varepsilon}\sqrt{1-\frac{\mu^{2}}{4}}+\frac{1}{2}.\label{eq:evs Tovbis}
\end{equation}
Examples of the absolute error of the SPPS method for various values
of $\mu$ and $\epsilon$ for this example are given in Table \ref{eq: Tabla String 1-1}.

\begin{table}[tbh]
\centering%
\begin{tabular}{ccccc}
\hline
$\mu$  & $\epsilon$  & %
\begin{tabular}{c}
Number of\tabularnewline
eigenvalues\tabularnewline
according to \eqref{eq:evs Tovbis}\tabularnewline
\end{tabular} & Eigenvalues  & Abs. error\tabularnewline
\hline
$0.5$  & $0.5$  & $2$  & %
\begin{tabular}{c}
$0.218245870595339i$\tabularnewline
$0.718245836551895i$\tabularnewline
\end{tabular} & %
\begin{tabular}{c}
$3.4\cdot10^{-8}$\tabularnewline
$5.5\cdot10^{-14}$\tabularnewline
\end{tabular}\tabularnewline
\hline
$1$  & $0.4$  & $2$  & %
\begin{tabular}{c}
$0.266025403665342i$\tabularnewline
$0.666025403780913i$\tabularnewline
\end{tabular} & %
\begin{tabular}{c}
$1.6\cdot10^{-11}$\tabularnewline
$1.8\cdot10^{-12}$\tabularnewline
\end{tabular}\tabularnewline
\hline
$0.5$  & $0.3$  & $3$  & %
\begin{tabular}{c}
$0.218245836882759i$\tabularnewline
$0.518245836565527i$\tabularnewline
$0.818245836562013i$\tabularnewline
\end{tabular} & %
\begin{tabular}{c}
$1.2\cdot10^{-11}$\tabularnewline
$9.3\cdot10^{-12}$\tabularnewline
$6.9\cdot10^{-12}$\tabularnewline
\end{tabular}\tabularnewline
\hline
$0.5$  & $0.2$  & $5$  & %
\begin{tabular}{c}
$0.068246698084900i$\tabularnewline
$0.268245836829048i$\tabularnewline
$0.468245836567762i$\tabularnewline
$0.668245836561404i$\tabularnewline
$0.868245836536161i$\tabularnewline
\end{tabular} & %
\begin{tabular}{c}
$2.9\cdot10^{-9}$\tabularnewline
$2.7\cdot10^{-10}$\tabularnewline
$1.4\cdot10^{-9}$\tabularnewline
$8.2\cdot10^{-11}$\tabularnewline
$3.4\cdot10^{-12}$\tabularnewline
\end{tabular}\tabularnewline
\hline
$1$  & $0.15$  & $6$  & %
\begin{tabular}{c}
$0.791025403810370i$\tabularnewline
$0.641025403330507i$\tabularnewline
$0.491025404649274i$\tabularnewline
$0.341019810999524i$\tabularnewline
$0.191026126515883i$\tabularnewline
$0.041103889485868i$\tabularnewline
\end{tabular} & %
\begin{tabular}{c}
$7.3\cdot10^{-11}$\tabularnewline
$1.8\cdot10^{-10}$\tabularnewline
$3.5\cdot10^{-9}$\tabularnewline
$5.8\cdot10^{-6}$\tabularnewline
$1.0\cdot10^{-7}$\tabularnewline
$2.1\cdot10^{-5}$\tabularnewline
\end{tabular}\tabularnewline
\hline
$1$  & $0.13$  & $7$  & %
\begin{tabular}{c}
$0.801025400788594i$\tabularnewline
$0.671025403049713i$\tabularnewline
$0.541025417533800i$\tabularnewline
$0.411024735536490i$\tabularnewline
$0.281025134994191i$\tabularnewline
$0.151030141516677i$\tabularnewline
$0.028942251261828i$\tabularnewline
\end{tabular} & %
\begin{tabular}{c}
$6.7\cdot10^{-8}$\tabularnewline
$3.1\cdot10^{-9}$\tabularnewline
$1.3\cdot10^{-8}$\tabularnewline
$7.8\cdot10^{-5}$\tabularnewline
$1.3\cdot10^{-7}$\tabularnewline
$1.0\cdot10^{-6}$\tabularnewline
$2.7\cdot10^{-3}$\tabularnewline
\end{tabular}\tabularnewline
\hline
$1$  & $0.12$  & $7$  & %
\begin{tabular}{c}
$0.806025211476463i$\tabularnewline
$0.686025403243744i$\tabularnewline
$0.566025442436003i$\tabularnewline
$0.446080832925892i$\tabularnewline
$0.326025177679161i$\tabularnewline
$0.206028821238862i$\tabularnewline
$0.085971656937961i$\tabularnewline
\end{tabular} & %
\begin{tabular}{c}
$4.0\cdot10^{-6}$\tabularnewline
$8.8\cdot10^{-9}$\tabularnewline
$1.3\cdot10^{-7}$\tabularnewline
$2.7\cdot10^{-6}$\tabularnewline
$2.8\cdot10^{-6}$\tabularnewline
$5.9\cdot10^{-6}$\tabularnewline
$8.0\cdot10^{-5}$\tabularnewline
\end{tabular}\tabularnewline
\hline
$0.5$  & $0.12$  & $8$  & %
\begin{tabular}{c}
$0.913160602214353i$\tabularnewline
$0.788245879788585i$\tabularnewline
$0.668245834591005i$\tabularnewline
$0.548245717744945i$\tabularnewline
$0.428227707506773i$\tabularnewline
$0.308246463131394i$\tabularnewline
$0.188245728430750i$\tabularnewline
$0.068265536390608i$\tabularnewline
\end{tabular} & %
\begin{tabular}{c}
$1.1\cdot10^{-2}$\tabularnewline
$1.3\cdot10^{-7}$\tabularnewline
$1.8\cdot10^{-4}$\tabularnewline
$8.5\cdot10^{-8}$\tabularnewline
$2.6\cdot10^{-5}$\tabularnewline
$6.7\cdot10^{-7}$\tabularnewline
$1.0\cdot10^{-6}$\tabularnewline
$1.2\cdot10^{-5}$\tabularnewline
\end{tabular}\tabularnewline
\hline
\end{tabular}\caption{Eigenvalue errors from Example \ref{Example String 1-1-1}.}
\label{eq: Tabla String 1-1}
\end{table}
\end{example}

\section{Relationship with the Dirac system}

A relation between the one-dimensional stationary Dirac system from
relativistic quantum theory and a quadratic Sturm-Liouville pencil
can be established as follows. Consider the following canonical form
of the Dirac system (see, e.g., \cite[section 7]{Sturm-Liouville Dirac})
\begin{align*}
y_{2}^{\prime}+\left(v(x)+\lambda\right)y_{1} & =Ey_{1},\\
-y_{1}^{\prime}+(v(x)-\lambda)y_{2} & =Ey_{2}.
\end{align*}
Here $v$ is a potential, $\lambda$ corresponding to the mass of
a particle is a spectral parameter and the constant $E$ is a fixed
energy. Adding and subtracting the equations one obtains the system
\begin{align*}
u^{\prime}+\left(v-E\right)w & =\lambda u,\\
w^{\prime}-\left(v-E\right)u & =-\lambda w
\end{align*}
for the functions $u=y_{2}-y_{1}$ and $w=y_{2}+y_{1}$. Supposing
$v-E\neq0$ on the domain of interest, it is easy to see from the
second equality that $u=\frac{\lambda w+w^{\prime}}{v-E}$. Substituting
this expression into the first equality one obtains an equation of
the form (\ref{eq: S-L Bundle}):
\[
\left(\frac{w^{\prime}}{v-E}\right)^{\prime}+\left(v-E\right)w=\lambda^{2}\frac{w}{v-E}+\lambda\left(\frac{1}{v-E}\right)^{\prime}w.
\]
Analogously to what was presented in the preceding section for the
Zakharov-Shabat system the SPPS representation for the solutions of
the one-dimensional Dirac system as well as for the characteristic
equations of corresponding spectral problems can be obtained.

\section{A spectral problem for the equation of a smooth string with a distributed
friction}

The equation for the transverse displacement $u(x,t)$ of a string
in an inhomogeneous absorbing medium (see, e.g., \cite{Atkinson,Jaulent,Kobyakova})
which extends in the $x$-direction from $x=0$ to $x=l$ with
the characteristic parameters $c(x)>0$, $b(x)>0$ and the characteristic
of absorption $\Gamma(x)$, has the form
\begin{equation}
\frac{\partial}{\partial x}\left(c(x)\frac{\partial u}{\partial x}\right)-b(x)\frac{\partial^{2}u}{\partial t^{2}}-\Gamma(x)\frac{\partial u}{\partial t}=0.\label{eq: inom string}
\end{equation}
The following condition expresses the hypothesis that the medium is
totally reflecting at $z=0$:
\begin{equation}
u(0,t)=0.\label{eq: inom string cond}
\end{equation}
For a wave of frequency $\lambda$, i.e., for $u(x,t)=y(\lambda,x)e^{-i\lambda t}$,
(\ref{eq: inom string}) and (\ref{eq: inom string cond}) take the
form
\begin{equation}
\begin{cases}
\dfrac{\partial}{\partial x}\left(c(x)\dfrac{\partial y}{\partial x}\right)+\lambda^{2}b(x)y+i\lambda\Gamma(x)y=0,\\
y(\lambda,0)=0.
\end{cases}\label{DampString}
\end{equation}
The equation in (\ref{DampString}) is of the form (\ref{eq: S-L Bundle}).
For numerical examples we will consider $c\equiv1$, and the absorption
$i\Gamma(x)=:2a(x)$ , i.e.,
\begin{equation}
y^{\prime\prime}=2a(x)\lambda y+b(x)\lambda^{2}y\label{eq: String}
\end{equation}
with the boundary conditions
\begin{equation}
y(0)=y(1)=0.\label{eq: condiciones frontera cuerda}
\end{equation}

\subsection{Solution using the SPPS method}

Equation (\ref{eq: String}) is a quadratic Sturm-Liouville pencil
with $p=1$, $q=0$, $r_{1}=2a$ and $r_{2}=b$.

Take $x_{0}=0$ and $u_{0}=1$ for the SPPS method of Theorem \ref{thm:SL grado N}.
Due to the boundary conditions (\ref{eq: condiciones frontera cuerda})
we have that $c_{1}=0$. Then the characteristic equation of the problem
(\ref{eq: String}), (\ref{eq: condiciones frontera cuerda}) has
the form
\[
\sum_{n=0}^{\infty}\lambda^{n}X^{\left(2n+1\right)}(1)=0,
\]
and hence we are interested in locating zeros of the polynomial
\begin{equation}
\Phi_{M}\left(\lambda\right)=\sum_{n=0}^{M}\lambda^{n}X^{\left(2n+1\right)}(1)\label{Poly}
\end{equation}
which approximate the eigenvalues of the problem.

In the following numerical examples the recursive integrals $X^{(n)}$ and $\widetilde{X}^{(n)}$ were calculated
using the Newton-Cottes 6 point integration formula of 7-th order.
On each step for the integration (if not specified explicitly) we took 100000 equally spaced sampling points on the segment $[0,1]$.

\begin{example} \label{Example String 1}Consider equation (\ref{eq: String})
with $a=1$ and $b=1$. This is the case of a constant damping of
a vibrating string. The exact characteristic equation \cite{Cox-Zuazua}
has the form $\lambda^{2}+2\lambda=-n^{2}\pi^{2}$, that is,
\begin{equation*}
\lambda_{\pm n}=-1\pm\sqrt{1-n^{2}\pi^{2}},\qquad n=1,2,\ldots.
\end{equation*}
Approximation of the roots of the polynomial (\ref{Poly}) by means
of the routine \texttt{roots} of Matlab, with the machine precision
arithmetic and $M=100$ delivers the results presented in the second
column of Table \ref{eq: Tabla String 1}. The same procedure but
with the 256-digit precision arithmetic in Mathematica delivers the
results presented in the third column of Table \ref{eq: Tabla String 1}.
As it can be appreciated, the first several eigenvalues are computed
with a considerably better accuracy meanwhile the accuracy of higher
computed eigenvalues does not change significantly. Doubling the number of the used formal powers delivers twice as many eigenvalues preserving the precision of the first eigenvalues and improving the precision of the forthcoming ones.

\begin{table}[tbh]
\centering %
\begin{tabular}{ccccc}
\hline
Eigenvalue  & %
\begin{tabular}{c}
Abs. error,\tabularnewline
machine prec. \tabularnewline
\end{tabular} & %
\begin{tabular}{c}
Abs. error,\tabularnewline
256-digit prec. \tabularnewline
($M=100$)
\end{tabular} & %
\begin{tabular}{c}
Abs. error,\tabularnewline
256-digit prec. \tabularnewline
($M=200$)
\end{tabular} & %
\begin{tabular}{c}
Abs. error,\tabularnewline
machine precision\tabularnewline
with spectral shift\tabularnewline
($M=100$) \tabularnewline
\end{tabular}\tabularnewline
\hline
$\lambda_{\pm1}$  & $1.3\cdot10^{-12}$  & $1.4\cdot10^{-29}$ & $1.4\cdot10^{-29}$ & $6.2\cdot10^{-13}$  \tabularnewline
$\lambda_{\pm2}$  & $1.8\cdot10^{-13}$  & $1.8\cdot10^{-27}$ & $1.8\cdot10^{-27}$ & $3.0\cdot10^{-13}$  \tabularnewline
$\lambda_{\pm3}$  & $3.9\cdot10^{-13}$  & $3.0\cdot10^{-26}$ & $3.0\cdot10^{-26}$ & $4.0\cdot10^{-14}$  \tabularnewline
$\lambda_{\pm4}$  & $4.5\cdot10^{-12}$  & $2.3\cdot10^{-25}$ & $2.3\cdot10^{-25}$ & $4.1\cdot10^{-13}$  \tabularnewline
$\lambda_{\pm5}$  & $1.7\cdot10^{-10}$  & $1.1\cdot10^{-24}$ & $1.1\cdot10^{-24}$ & $1.6\cdot10^{-13}$  \tabularnewline
$\lambda_{\pm6}$  & $3.7\cdot10^{-9}$  & $3.8\cdot10^{-24}$ & $3.8\cdot10^{-24}$ & $2.1\cdot10^{-13}$  \tabularnewline
$\lambda_{\pm7}$  & $3.6\cdot10^{-8}$  & $1.1\cdot10^{-23}$ & $1.1\cdot10^{-23}$ & $2.5\cdot10^{-13}$  \tabularnewline
$\lambda_{\pm8}$  & $3.3\cdot10^{-7}$  & $2.5\cdot10^{-20}$ & $2.9\cdot10^{-23}$  & $2.6\cdot10^{-13}$  \tabularnewline
$\lambda_{\pm9}$  & $1.9\cdot10^{-5}$  & $4.5\cdot10^{-15}$ & $6.6\cdot10^{-23}$ & $1.8\cdot10^{-13}$  \tabularnewline
$\lambda_{\pm10}$  & $4\cdot10^{-4}$  & $2.3\cdot10^{-10}$ & $1.4\cdot10^{-22}$ & $2.4\cdot10^{-13}$  \tabularnewline
$\lambda_{\pm11}$  & $4\cdot10^{-3}$  & $4.2\cdot10^{-6}$ & $2.7\cdot10^{-22}$ & $3.0\cdot10^{-13}$  \tabularnewline
$\lambda_{\pm12}$  & $3.1\cdot10^{-2}$  & $3.6\cdot10^{-2}$ & $4.9\cdot10^{-22}$ & $3.0\cdot10^{-13}$  \tabularnewline
$\lambda_{\pm13}$  & $1.5$  & $1.9$ & $8.6\cdot10^{-22}$ & $1.8\cdot10^{-13}$  \tabularnewline
$\lambda_{\pm15}$  &  &  & $2.3\cdot10^{-21}$ & $1.2\cdot10^{-13}$  \tabularnewline
$\lambda_{\pm20}$  &  & & $2.4\cdot10^{-17}$ & $7.1\cdot10^{-14}$ \tabularnewline
$\lambda_{\pm25}$  &  & & $2.9$ & $1.6\cdot10^{-13}$  \tabularnewline
$\lambda_{\pm35}$  &  & & & $9.0\cdot10^{-13}$  \tabularnewline
$\lambda_{\pm50}$  &  & & & $1.9\cdot10^{-11}$  \tabularnewline
$\lambda_{\pm70}$  &  & & & $7.8\cdot10^{-10}$  \tabularnewline
$\lambda_{\pm100}$  &  & & & $4.0\cdot10^{-7}$  \tabularnewline
\hline
\end{tabular}\caption{Eigenvalue errors from Example \ref{Example String 1}.}
\label{eq: Tabla String 1}
\end{table}

This shows that the usage of arbitrary precision arithmetic allows
one to approximate more accurately the formal powers and the first eigenvalues, to obtain several additional eigenvalues. Meanwhile the spectral shift technique described in Section \ref{Spectral shift}
permits to enhance the accuracy of the obtained eigenvalues as well
as to calculate higher order eigenvalues even in machine precision.

\begin{figure}[tbh]
\centering
\includegraphics[bb=108 306 504 486,width=5.5in,height=2.5in]
{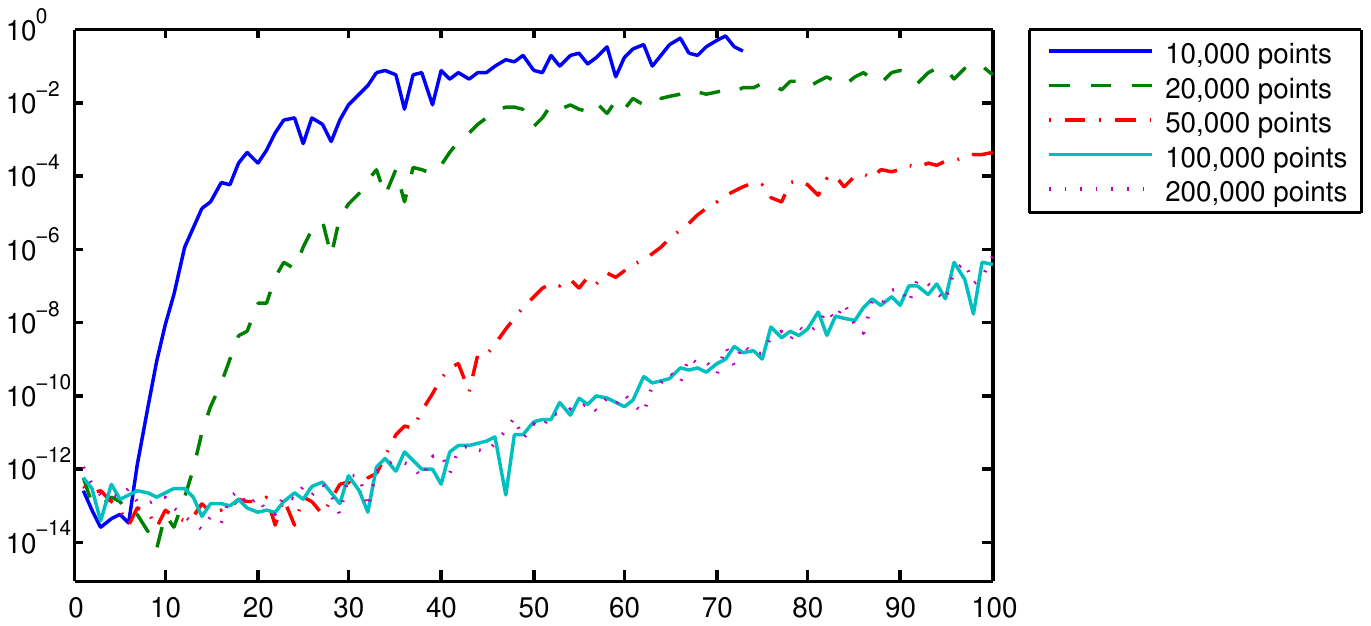}\\
\includegraphics[bb=108 306 504 486,width=5.5in,height=2.5in]
{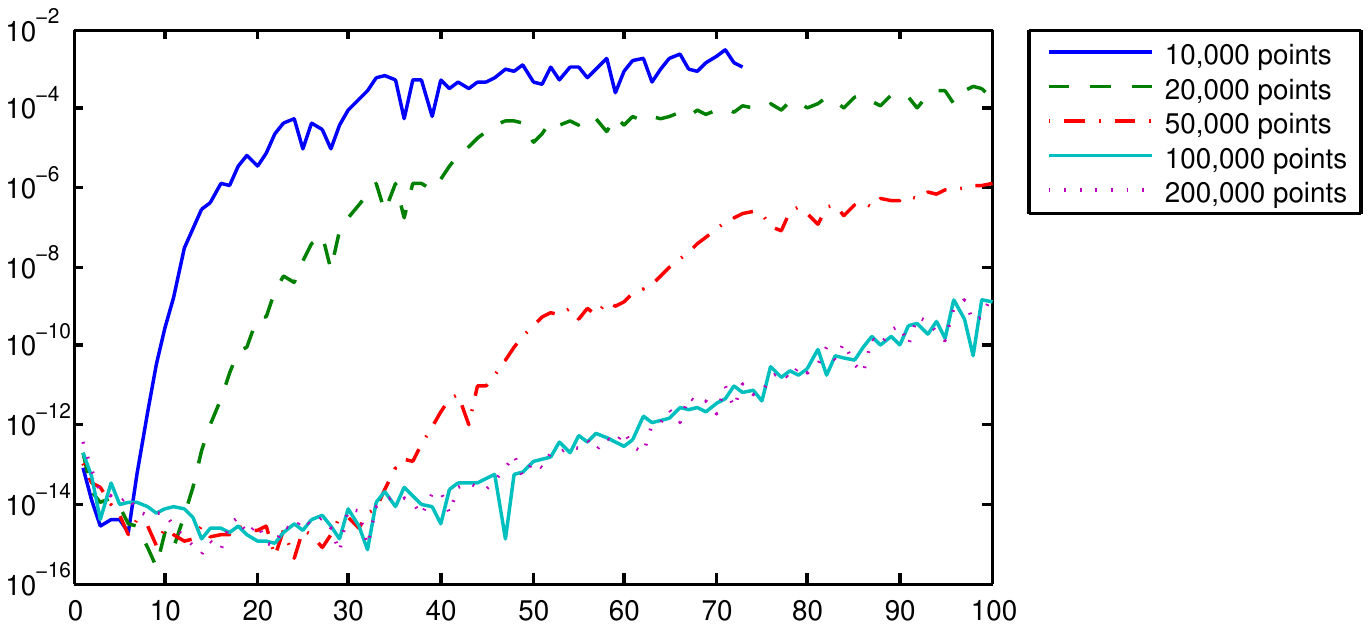}
\caption{The plots of absolute (above) and relative (below) errors of the eigenvalues from Example \ref{Example String 1} computed by means of the SPPS method using different numbers of points for calculating recursive integrals. The axis of abscissas corresponds to the ordinal number of the eigenvalue.}
\label{Fig1}
\end{figure}

We used values  $\widetilde\lambda_{n}=-1-(0.1+3i)n$ for the spectral shifts, on each step computing 201 formal powers, i.e., we took $M=100$ in \eqref{Poly}, and evaluating the new particular solution in terms of these formal powers and the SPPS representation. The roots of \eqref{Poly} closest to the current spectral shift were stored as approximate eigenvalues. The results obtained with the spectral shift procedure are presented in the fifth column of Table \ref{eq: Tabla String 1}. The eigenvalues were computed using the machine precision. A significant improvement in the accuracy and in the number of the found eigenvalues can be appreciated.

We verified the dependence of the eigenvalue errors on the number $N$ of points used for calculation of the recursive integrals. The plots of the absolute and relative errors of the eigenvalues for different values of $N$ obtained using the spectral shift technique described above are presented on Fig.\ \ref{Fig1}. The increment of the value of $N$ to 100000 leads to the improvement of the eigenvalues accuracy, meanwhile further increment of $N$ does not change it significantly. The slow growth of the error for $N=100000$ is due to the increasing distance between the values $\widetilde\lambda_{n}$ used for the spectral shift and the eigenvalues. The observed rapid growth of the error starting from some particular eigenvalue index for smaller values of $N$  can be explained recalling that the higher index eigenfunctions as well as the solutions for close values of the spectral parameter are highly oscillatory, i.e., have large derivatives, and taking into account the error formula of Newton-Cottes integration rule (see \cite[\S2.4 and (2.5.26)]{Rabinowitz}).
\end{example}

In the next example we present a problem with variable coefficients and  illustrate the dependence of the approximate eigenvalues precision on the truncation parameter $M$ in \eqref{Poly}.

\begin{example} \label{Example String 2}Consider equation (\ref{eq: String})
with $a=x^{2}$ and $b=1$. The exact characteristic equation for
this problem is
\begin{equation}
\begin{split} & 2^{\frac{\lambda^{\frac{3}{2}}}{2\sqrt{2}}}\Gamma\left(\frac{1}{8}\left(6+\sqrt{2}\lambda^{\frac{3}{2}}\right)\right)\operatorname{D}_{\frac{1}{4}\left(-2-\sqrt{2}\lambda^{\frac{3}{2}}\right)}\left(2^{\frac{3}{4}}\lambda^{\frac{1}{4}}\right)\\
 & \qquad-\Gamma\left(\frac{1}{8}\left(6-\sqrt{2}\lambda^{\frac{3}{2}}\right)\right)\operatorname{D}_{\frac{1}{4}\left(-2+\sqrt{2}\lambda^{\frac{3}{2}}\right)}\left(2^{\frac{3}{4}}i\lambda^{\frac{1}{4}}\right)=0,
\end{split}
\label{DispEq1}
\end{equation}
where $\operatorname{D}$ is the parabolic cylinder function. For
comparing the numerical results obtained by means of the SPPS method
with the exact eigenvalues, Wolfram's Mathematica \texttt{FindRoot}
command was used for calculating the roots of (\ref{DispEq1}).

On Fig.\ \ref{Fig2} we present the plots of the absolute
errors of the computed eigenvalues using different values of $M$. All computations were performed in Matlab using machine precision, $100000$ points for evaluating recursive integrals and applying spectral shift technique with the spectral shifts $\lambda_n=-1-4\pi n i$, $n\le 30$. As can be seen from the presented plots, the truncation parameter $M$ in \eqref{Poly} strongly affects the precision after the first spectral shift, meanwhile for the subsequent spectral shifts the precision is preserved. Starting from some particular value of $M$ the eigenvalue precision almost does not change, there is a small difference between $M=40$ and $M=50$ and no visible difference for $M=60$ (we did not include the errors for $M=60$ on the plots for this reason).

\begin{figure}[tbh]
\centering
\includegraphics[bb=108 301 504 489, width=5.5in,height=2.6in]
{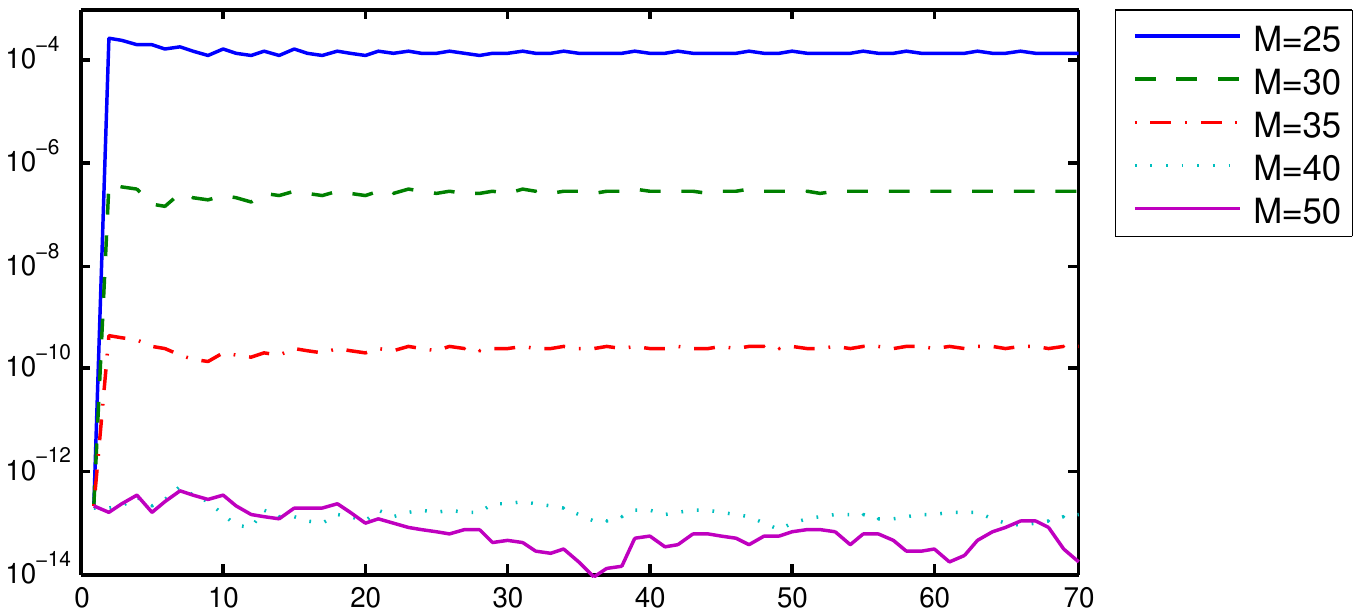}\caption{Absolute errors of the eigenvalues from Example \ref{Example String 2} computed by means of the SPPS method using different number of formal powers. The axis of abscissae corresponds to the ordinal number of the eigenvalue.}
\label{Fig2}
\end{figure}

\end{example}

\end{document}